\documentclass[reqno,11pt]{amsart}
\usepackage[colorlinks=true, linkcolor=blue, citecolor=blue]{hyperref}
    
\usepackage{amssymb}
\usepackage{amsmath, graphicx, rotating}
\usepackage{color}
\usepackage{soul}
\usepackage[dvipsnames]{xcolor}
               
\usepackage{ifthen}      
\usepackage{xkeyval}
\usepackage{todonotes}     
\usepackage[T1]{fontenc}
\usepackage{lmodern}
\usepackage[english]{babel}

\usepackage{ upgreek }
\usepackage{stmaryrd}
\SetSymbolFont{stmry}{bold}{U}{stmry}{m}{n}
\usepackage{amsthm}
\usepackage{float}

\usepackage{ bbm }
\usepackage{ stmaryrd }
\usepackage{ mathrsfs }
\usepackage{ frcursive }
\usepackage{ comment }

\usepackage{pgf, tikz}
\usetikzlibrary{shapes}
\usepackage{varioref}
\usepackage{enumitem}

\usepackage{mathtools}  

\usepackage{dsfont}

\definecolor{rouge}{rgb}{0.7,0.00,0.00}
\definecolor{vert}{rgb}{0.00,0.5,0.00}
\definecolor{bleu}{rgb}{0.00,0.00,0.8}

\usepackage[margin=1.26in]{geometry}

\newtheorem{theorem}{Theorem}[section]
\newtheorem*{theorem*}{Theorem}
\newtheorem{lemma}[theorem]{Lemma}

\newtheorem{proposition}[theorem]{Proposition}

\labelformat{hypothesis}{\textbf{M\kern-0.1mm#1}}

\newtheorem{condition}{Condition}
\newtheorem{conditionA}{A\kern-0.1mm}
\labelformat{conditionA}{\textbf{A\kern-0.1mm#1}}

\renewcommand\dots{\hbox to 1em{.\hss.\hss.}}

\theoremstyle{definition}

\def \eref#1{\hbox{(\ref{#1})}}

\numberwithin{equation}{section}

\def\bb#1{\mathbb{#1}}

\def\scr#1{\mathscr{#1}}

\def\geq{\geqslant}
\def\leq{\leqslant}

\newcommand{\ee}{\varepsilon}

\DeclareMathOperator{\supp}{supp}

\DeclarePairedDelimiter\floor{\lfloor}{\rfloor}

\begin{document}

\title[Moderate deviations and local limit theorems]
{Moderate deviations and local limit theorems for the coefficients of random walks on the \\ general linear group$^1$}
\footnote{This paper is a part of the results which previously appeared in Xiao, Grama, Liu "Limit theorems for the coefficients of random walks on the general linear group" arXiv:2111.10569, 2021}

\author{Hui Xiao}
\author{Ion Grama}
\author{Quansheng Liu}

\curraddr[Xiao, H.]{Universit\"at Hildesheim, Institut f\"ur Mathematik und Angewandte Informatik, Hildesheim, Germany}
\email{xiao@uni-hildesheim.de}
\curraddr[Grama, I.]{ Universit\'{e} de Bretagne-Sud, LMBA UMR CNRS 6205, Vannes, France}
\email{ion.grama@univ-ubs.fr}
\curraddr[Liu, Q.]{ Universit\'{e} de Bretagne-Sud, LMBA UMR CNRS 6205, Vannes, France}
\email{quansheng.liu@univ-ubs.fr}



\begin{abstract}
Consider the random walk $G_n : = g_n \ldots g_1$, $n \geq 1$, 
where $(g_n)_{n\geq 1}$ is a sequence of independent and identically distributed 
random elements with law $\mu$ on the general linear group $\textup{GL}(V)$ with $V=\bb R^d$. 
Under suitable conditions on $\mu$,  we establish Cram\'{e}r type moderate deviation expansions
and local limit theorems with moderate deviations for the coefficients $\langle f, G_n v \rangle$, where $v \in V$ and $f \in V^*$.  
Our approach is based on the H\"older regularity of 
the invariant measure of the Markov chain $G_n \!\cdot \! x = \bb R G_n v$ on the projective space of $V$ 
with the starting point $x = \bb R v$, under the changed measure. 
\end{abstract}

\date{\today}

\maketitle



\section{Introduction} 
\subsection{Background and objectives}
There is growing interest in studying random walks on linear groups since
the groundwork of Furstenberg and Kesten \cite{FK60}, 
see also Le Page \cite{LeP82}, 
Guivarc'h and Raugi \cite{GR85},  Bougerol and Lacroix \cite{BL85}, Goldsheid and Margulis \cite{GM89},
Benoist and Quint \cite{BQ16b}, and the references therein. 
This theory has important applications in a number of research areas such as 
spectral theory \cite{BL85, CL90, BG12},  
geometric measure theory \cite{PT08, HS17, GK21}, 
statistical physics \cite{CPV93},   
homogeneous dynamics \cite{BFLM11, BQ13},
stochastic recursions and smoothing transforms \cite{Kes73, GL16, Men16},
and branching processes in random environment \cite{LPP18, GLP20a}.
These studies are often related to the asymptotic properties of the random walk
\begin{align*}
G_n : = g_n \ldots g_1,  \quad  n \geq 1, 
\end{align*}
where $(g_n)_{n \geq 1}$ is a sequence of independent and identically distributed
(i.i.d.) random elements with law $\mu$ on the general linear group $\textup{GL}(V)$ with $V = \bb R^d$. 
Of particular interest is the investigation of the growth rate of the coefficients $\langle f, G_n v \rangle$, 
where $v \in V$, $f \in V^*$ and $\langle \cdot, \cdot \rangle$ is the duality bracket: $\langle f, v \rangle = f(v)$.  
In this direction, Furstenberg and Kesten \cite{FK60} established the strong law of large numbers 
in the case of positive matrices, namely, 
$\lim_{n\to\infty} \frac{1}{n} \log | \langle f, G_n v \rangle | = \lambda_1,$ a.s., 
where $\lambda_1 \in \bb R$ is a constant (independent of $f$ and $v$) called the first Lyapunov exponent of $\mu$; 
see also Kingman \cite{Kin73}, 
Cohn, Nerman and Peligrad \cite{CNP93} and Hennion \cite{Hen97}. 
For invertible matrices, the law of large numbers with the corresponding Lyapunov exponent $\lambda_1$
 has been established by Guivarc'h and Raugi \cite{GR85}. 
The central limit theorem for the coefficients has also been established in \cite{GR85}  under the exponential moment condition on $\mu$
which has recently been relaxed to the optimal second moment condition by Benoist and Quint \cite{BQ16}: 
for any $t \in \bb R$, 
\begin{align}\label{Ch7_CLT_Entry0a}
\lim_{n \to \infty} \bb{P} \left( \frac{ \log | \langle f, G_n v \rangle | 
 - n \lambda_1}{ \sigma \sqrt{n} } \leq t  \right)  =  \Phi(t), 
\end{align}
where $\Phi$ is the standard normal distribution function on $\bb R$
and $\sigma^2 > 0$ is the asymptotic variance of $ \frac{1}{\sqrt n} \log | \langle f, G_n v \rangle |.$  
Further refinements of this result including the Berry-Esseen bound and the first-order Edgeworth expansion 
can be found in \cite{GR85, DKW21, CDMP21b, XGL22a}. 

In this paper, our first objective is to quantity the relative error in \eqref{Ch7_CLT_Entry0a} 
by establishing the Cram\'er type moderate deviation expansion under appropriate conditions on $\mu$: we prove that,
as $n \to \infty$,  uniformly in $t \in [0, o(\sqrt{n} )]$, $v \in V$ and $f \in V^*$ with $\|v\| = \|f\| =1$,  
\begin{align}
\frac{\bb{P} \Big( \frac{\log | \langle f, G_n v \rangle | 
- n\lambda_1}{ \sigma \sqrt{n} } \geq t \Big)} {1-\Phi(t)}
& =  e^{ \frac{t^3}{\sqrt{n}} \zeta ( \frac{t}{\sqrt{n}} ) } \big[ 1 +  o(1) \big], \label{Ch7_Cramer_Entry_0a}
\end{align}
where $\zeta$ is the Cram\'{e}r series, cf.\ \eqref{Ch7Def-CramSeri}. A similar expansion for the lower tail is also obtained.
More generally, we prove the Cram\'{e}r type moderate deviation expansion for the couple 
$(G_n \!\cdot\! x, \log |\langle f, G_n v \rangle|)$ with a target function $\varphi$ on the Markov chain $(G_n \!\cdot\! x)$
on the projective space $\bb P(V)$; see Theorem \ref{Thm-Cram-Entry_bb}. 
As information, for the norm cocycle $\log \frac{\| G_n v \|}{\|v\|}$, 
we note that the moderate deviation principle  has been proved in \cite{CDM17}
and  Cram\'er type moderate deviation expansion has been established in \cite{XGL19b}. 
For positive matrices, 
as a consequence of the expansion for the norm cocycle from \cite{XGL19b}, 
the corresponding result for the coefficients $\langle f, G_n v \rangle$ has been obtained in \cite{XGL19c}.  
It is worth mentioning that, in the case of invertible matrices, 
the proof of the expansion \eqref{Ch7_Cramer_Entry_0a} is much more involved and cannot be deduced from 
the corresponding result for the norm cocycle. 

Our second objective is to establish the  
local limit theorem with moderate deviations for the coefficients $\langle f, G_n v \rangle$: 
we prove that, for any real numbers $a_1 < a_2$, 
as $n \to \infty$, uniformly in $|t| = o(\sqrt{n})$, 
\begin{align}\label{LLT-MD-Intro}
\mathbb{P} \Big( \log| \langle f, G_n v \rangle | - n\lambda_1 \in [a_1, a_2] + \sqrt{n}\sigma t \Big) 
= \frac{a_2 - a_1 + o(1)}{ \sigma \sqrt{2 \pi n} } 
 e^{ - \frac{t^2}{2} + \frac{t^3}{\sqrt{n}}\zeta(\frac{t}{\sqrt{n}} ) }. 
\end{align}
In fact, as before, we will establish a more general local limit theorem 
 with moderate deviations ($|t| = o(\sqrt{n})$) for the couple $(G_n \!\cdot\! x, \log |\langle f, G_n v \rangle|)$,
 see Theorem \ref{ThmLocal02}.  
The Local limit theorem with large deviations ($|t| = c \sqrt{n}$)  
for the coefficients $\langle f, G_n v \rangle$ has been obtained in \cite{XGL19d}
using the Bahadur-Rao-Petrov type large deviation asymptotics. 
Local limit theorems with large and moderate deviations for the norm cocycle $\log \frac{\| G_n v \|}{\|v\|}$
have been established recently in \cite{BQ16b, XGL20-SPA, XGL19b}.

Finally we would like to mention that all the results of this paper remain valid when $V$ is $\bb C^d$ or $\bb K^d$, where 
$\bb K$ is any local field. 

\subsection{Proof strategy} 
Our strategy of the proof is based on the following decomposition
which relates the coefficients $\langle f, G_n v \rangle$ to the norm cocycle $\sigma (G_n, x) = \log \frac{\| G_n v \|}{\|v\|}$: 
for any $x = \bb R v \in \bb P(V)$ and $y = \bb R f \in \bb P(V^*)$ with $\|f\|=1$, 
writing $\delta(y, x) = \frac{|\langle f, v \rangle|}{\|f\| \|v\|}$, 
we have 
\begin{align}\label{Basic-scalar-decom}
\log |\langle f, G_n v \rangle| = \sigma (G_n, x) +  \log \delta(y, G_n \!\cdot\! x). 
\end{align}
To establish the Cram\'er type moderate deviation expansion \eqref{Ch7_Cramer_Entry_0a}, 
our approach is different from the standard one which is based on 
performing a change of measure and proving 
a Berry-Esseen bound under the changed measure; see for example Cram\'{e}r \cite{Cra38} and Petrov \cite{Pet75}. 
Note that, even with a Berry-Esseen bound for $\log |\langle f, G_n v \rangle|$ under the changed measure at hands,
we do not know how to obtain \eqref{Ch7_Cramer_Entry_0a}  
using this strategy. 
The main difficulty resides in the fact that, when we use a change of measure corresponding to
the norm cocycle $\sigma (G_n, x)$ in order to prove moderate deviations for $\log |\langle f, G_n v \rangle|$, 
the H\"older regularity of the error term $\log \delta(y, G_n \!\cdot\! x)$ in \eqref{Basic-scalar-decom}
is not enough to obtain the desired result. 
The approach in this paper
consists in decomposing the error $\log \delta(y, G_n \!\cdot\! x)$ 
into a sum of components using a partition $(\chi_{n,k}^y)_{k \geq 1}$ of the unity on the projective space $\bb P(V)$ 
 following the ideas from  \cite{DKW21, XGL22a}.
Then, for all components, we pass to the Fourier transforms under the changed measure, 
and establish their exact asymptotic expansions, see Propositions \ref{KeyPropo} and \ref{KeyPropo-02},
which are the key points of our proof.
We conclude by patching up these expansions using the exponential H\"older regularity 
of the invariant measure of the Markov chain $(G_n \!\cdot\! x)$ under the changed measure, which has been 
established recently in \cite{GQX20, XGL22a}. 
 The advantage of using the partition of the unity \cite{DKW21} 
is that each piece is a smooth compactly supported function
and therefore does not need to be smoothed additionally as in the previous works \cite{Xiao20, XGL19d}. 
In addition, patching  up this partition is more effective since there is no loss of mass due to the additional smoothing.
This simplifies the proofs and avoids the use of additional properties of the invariant measure like the zero-one 
law established in \cite{GQX20}.

The proof of the local limit theorem with moderate deviations \eqref{LLT-MD-Intro}
follows the same lines as the proof of the expansion \eqref{Ch7_Cramer_Entry_0a}. 
In the argument we use the uniform version of the exponential H\"older regularity 
of the invariant measure of the Markov chain $(G_n \!\cdot\! x)$
under the changed measure.


\section{Main results}

\subsection{Notation and conditions}
Consider the $d$-dimensional Euclidean space $V = \bb R^d$, where $d \geq 1$ is an integer. 
Let $e_1, \ldots, e_d$  be  a basis  of $V$.  
The norm on $V$ is denoted by $\|v\|^2 = \sum_{i=1}^d |v_i|^2$ for $v = \sum_{i=1}^d v_i e_i \in V$. 
Denote by $V^*$ the dual vector space of $V$ and by $e_1^*, \ldots, e_d^*$
the dual basis, so that $e_i^*(e_j)= 1$ if $i=j$ and  $e_i^*(e_j)= 0$ if $i\neq j$. 
Let $\wedge^2 V$ be the exterior product of $V$. 
We use the same symbol $\| \cdot \|$ for the norms induced on $\wedge^2 V$ and $V^*$.  
The projective space $\bb P (V)$ is equipped with the angular distance 
\begin{align}\label{Angular-distance}
d(x, x') = \frac{\| v \wedge v' \|}{ \|v\| \|v'\| }  \quad \mbox{for} \  x= \bb R v \in \bb P(V), \  x' = \bb R v' \in \bb P(V). 
\end{align}
The dual bracket is defined by $\langle f, v \rangle = f(v)$ for any $v \in V$ and $f \in V^*$. 
Denote
\begin{align*}
\delta(x,y) = \frac{| \langle f, v \rangle |}{\|f\| \|v\| }  \quad \mbox{for} \  x= \bb R v \in \bb P(V),  \  y = \bb R f \in \bb P(V^*).
\end{align*}
Let $\mathscr{C}(\bb{P}(V) )$ be the space of complex-valued continuous functions on $\bb{P}(V)$, 
equipped with the norm $\|\varphi\|_{\infty}: =  \sup_{x\in \bb{P}(V) } |\varphi(x)|$ for $\varphi \in \mathscr{C}(\bb{P}(V) )$.
For $\gamma>0$,  set
\begin{align*}
\|\varphi\|_{\gamma}: =  \|\varphi\|_{\infty} + 
  \sup_{x, x' \in \bb{P}(V): x \neq x'} \frac{|\varphi(x)-\varphi(x')|}{ d(x, x')^{\gamma} }.  
\end{align*}
The Banach space of complex-valued $\gamma$-H\"older continuous functions on $\bb P(V)$ is denoted by
\begin{align*}
\mathscr{B}_{\gamma}:= \left\{ \varphi\in \mathscr{C}(\bb P(V)): \|\varphi\|_{\gamma}< \infty \right\}. 
\end{align*}
The set of all bounded linear operators from $\mathscr{B}_{\gamma}$ to $\mathscr{B}_{\gamma}$,
 equipped with the operator norm
$\left\| \cdot \right\|_{\mathscr{B}_{\gamma} \to \mathscr{B}_{\gamma}}$, 
is denoted by $\mathscr{L(B_{\gamma},B_{\gamma})}$. 
The topological dual of $\mathscr B_\gamma$, denoted by $\mathscr B'_\gamma$, 
is endowed with the induced norm.   
All over the paper, we denote by $c, C$ positive constants whose values may change from line to line.

Let $\textup{GL}(V)$ be the general linear group of the vector space $V$.
We consider a Borel probability measure $\mu$ on $\textup{GL}(V)$.  
We denote the action of  $g \in \textup{GL}(V)$ on a vector $v \in V$ by $gv$, 
and the action of $g \in \textup{GL}(V)$ on a projective line $x = \bb R v \in \bb P(V)$ by $g \cdot x = \bb R gv$. 
For any $g \in \textup{GL}(V)$, let $\| g \| = \sup_{v \in V \setminus \{0\} } \frac{\| g v \|}{\|v\|}$
and $N(g) = \max \{ \|g\|, \| g^{-1} \| \}$. 
We need the following exponential moment condition. 

\begin{conditionA}
\label{Ch7Condi-Moment} 
There exists a constant $\ee >0$ such that $\int_{ \textup{GL}(V) } N(g)^{\ee} \mu(dg) < \infty$. 
\end{conditionA}

Denote by $\Gamma_{\mu}$ the smallest closed subsemigroup
generated by the support of the measure $\mu$. 
We say that an endomorphism $g$ of $V$ is proximal 
if it has an eigenvalue $\lambda$ with multiplicity one and all other eigenvalues of $g$ have modulus strictly less than $|\lambda|$.
Introduce the following strong irreducibility and proximality condition. 

\begin{conditionA}\label{Ch7Condi-IP}
{\rm (i)(Strong irreducibility)} 
No finite union of proper subspaces of $V$ is $\Gamma_{\mu}$-invariant.

{\rm (ii)(Proximality)}
$\Gamma_{\mu}$ contains a proximal endomorphism. 
\end{conditionA}

The norm cocycle $\sigma: \textup{GL}(V) \times \bb P(V) \to \bb R$ is defined by
\begin{align*}
\sigma (g, x) = \log \frac{\|gv\|}{\|v\|}, \quad   \mbox{for any} \  g \in \textup{GL}(V)  \  \mbox{and}  \   x = \bb R v \in \bb P(V). 
\end{align*}
Recall that $\lambda_1 \in \bb R$ is the first Lyapunov exponent of $\mu$. 
According to Le Page \cite[Theorem 2]{LeP82}, 
under \ref{Ch7Condi-Moment} and \ref{Ch7Condi-IP},
the limit
\begin{align}\label{Def-sigma}
\sigma^2: = \lim_{n \to \infty} \frac{1}{n} \bb E \left[ (\sigma (G_n, x) - n \lambda_1)^2 \right] \in (0, \infty),  
\end{align}
exists and is independent of $x \in \bb P(V)$. 
For any $s \in (-s_0, s_0)$ with $s_0 >0$ small enough, 
and any bounded measurable function $\varphi$ on $\bb P(V)$, define
\begin{align}\label{Def_Ps001}
P_s \varphi(x) = \int_{ \textup{GL}(V) } e^{s \sigma(g, x)} \varphi(g \!\cdot\! x) \mu(dg),  
\quad  x \in \bb P(V). 
\end{align}
It is known that
the transfer operator $P_s \in \mathscr{L(B_{\gamma},B_{\gamma})}$ 
has a unique dominant eigenvalue $\kappa(s)$ 
with $\kappa(0) = 1$ and the mapping $s \mapsto \kappa(s)$ being analytic, see Lemma \ref{Ch7transfer operator}. 

Let $\Lambda = \log \kappa$ and $\gamma_m = \Lambda^{(m)}(0)$ for $m \geq 1$.
In particular, it holds that $\gamma_1 = \lambda_1$ and $\gamma_2 = \sigma^2$. 
In the whole paper, we write $\zeta$ for the Cram\'{e}r series \cite{Pet75}: 
\begin{align}\label{Ch7Def-CramSeri}
\zeta(t)=\frac{\gamma_3}{6\gamma_2^{3/2} } + \frac{\gamma_4\gamma_2-3\gamma_3^2}{24\gamma_2^3}t
+ \frac{\gamma_5\gamma_2^2-10\gamma_4\gamma_3\gamma_2 + 15\gamma_3^3}{120\gamma_2^{9/2}}t^2 + \cdots, 
\end{align}
which converges for  $|t|$ small enough.

Under conditions \ref{Ch7Condi-Moment} and \ref{Ch7Condi-IP},
the Markov chain $(G_n \!\cdot\! x)_{n \geq 0}$ has a unique invariant probability measure 
$\nu$ on $\bb P(V)$ such that for any bounded measurable function $\varphi$ on $\bb P(V)$,
\begin{align} \label{Ch7mu station meas}
\int_{\bb P(V)} \int_{\textup{GL}(V)} \varphi(g \!\cdot\! x) \mu(dg) \nu(dx) 
 = \int_{ \bb P(V) } \varphi(x) \nu(dx)
= : \nu(\varphi).
 \end{align}

\subsection{Moderate deviation expansions}
In this subsection 
we state the following Cram\'{e}r type moderate deviation expansions for 
the coefficients $\langle f, G_n v \rangle$, and more generally, 
for the couple $(G_n \!\cdot\! x, \log |\langle f, G_n v \rangle|)$ with a target function $\varphi$ 
on the Markov chain $(G_n \!\cdot\! x)_{n \geq 0}$. 

\begin{theorem}\label{Thm-Cram-Entry_bb}
Assume \ref{Ch7Condi-Moment} and \ref{Ch7Condi-IP}. 
Then, 
there exists a constant $\gamma >0$ such that for any $\varphi \in \mathscr{B}_{\gamma}$, 
we have, as $n \to \infty$, uniformly in $t \in [0, o(\sqrt{n} )]$, 
$x = \bb R v \in \bb P(V)$ and $f \in V^*$ with $\|v\| = \|f\| =1$,
\begin{align}
\frac{\bb{E} \left[ \varphi(G_n \!\cdot\! x) \mathds{1}_{ 
\left\{ \log| \langle f,  G_n v \rangle | - n \lambda_1 \geq \sqrt{n} \sigma t \right\} } 
\right] }  { 1-\Phi(t) }    
& =  e^{ \frac{t^3}{\sqrt{n}} \zeta(\frac{t}{\sqrt{n}} ) } \big[ \nu(\varphi) +  o(1) \big],  
\label{LD general upper001} \\
\frac{\bb{E} \left[ \varphi(G_n \!\cdot\! x) \mathds{1}_{ 
\left\{ \log| \langle f,  G_n v \rangle | - n \lambda_1 \leq - \sqrt{n} \sigma t  \right\} } \right] } { \Phi(-t)  }   
& =  e^{ - \frac{t^3}{\sqrt{n}} \zeta(-\frac{t}{\sqrt{n}} ) } \big[ \nu(\varphi) +  o(1) \big]. 
\label{LD general lower001}  
\end{align}
\end{theorem}

We mention that the uniformity in $t \in [0, o(\sqrt{n} )]$ means that for each sequence $a_n = o(\sqrt{n})$ of positive numbers,
the conclusion holds uniformly for all $t \leq a_n$. 
Accordingly, the reminder term $o(1)$ in 
 \eqref{LD general upper001} and \eqref{LD general lower001} may depend on 
the sequence $(a_n)$. 
The expansion \eqref{Ch7_Cramer_Entry_0a} follows from \eqref{LD general upper001} by taking $\varphi=1$.

Theorem \ref{Thm-Cram-Entry_bb} clearly implies the following moderate deviation principle
for the couple $(G_n \!\cdot\! x, \log |\langle f, G_n v \rangle|)$ 
with a target function $\varphi$ on the Markov chain $(G_n \!\cdot\! x)$:
under \ref{Ch7Condi-Moment} and \ref{Ch7Condi-IP}, 
for any sequence of positive numbers $(b_n)_{n\geq 1}$ satisfying
$\frac{b_n}{n} \to 0$ and $\frac{b_n}{\sqrt{n}} \to \infty$, 
any Borel set $B \subseteq \bb R$
and any real-valued function $\varphi \in \mathscr{B}_{\gamma}$ satisfying $\nu(\varphi) > 0$,  
we have that uniformly in $x = \bb R v \in \bb P(V)$, $v \in V$ and $f \in V^*$ with $\|v\| = \|f\| =1$, 
\begin{align*} 
-\inf_{t \in B^{\circ}} \frac{t^2}{2\sigma^2} 
& \leq \liminf_{n\to \infty} \frac{n}{b_n^{2}}
\log  \bb{E} \Big[  \varphi( G_n \!\cdot\! x )  
\mathds{1}_{  \big\{ \frac{ \log |\langle f, G_n v \rangle| - n\lambda_1 }{b_n} \in B  \big\} }  \Big]
\nonumber\\
& \leq   \limsup_{n\to \infty}\frac{n}{b_n^{2}}
\log  \bb{E}  \Big[   \varphi( G_n \!\cdot\! x ) 
\mathds{1}_{ \big\{ \frac{ \log |\langle f, G_n v \rangle| - n\lambda_1 }{b_n} \in B  \big\} }   \Big] 
\leq - \inf_{t \in \bar{B}} \frac{t^2}{2\sigma^2},
\end{align*}
where $B^{\circ}$ and $\bar{B}$ are respectively the interior and the closure of $B$.
This moderate deviation principle is new even for $\varphi = 1$.

\subsection{Local limit theorem with moderate deviations}
In this subsection we state 
the local limit theorem with moderate deviations and target functions
for the coefficients $\langle f, G_n v \rangle$.

\begin{theorem}\label{ThmLocal02}
Assume \ref{Ch7Condi-Moment} and \ref{Ch7Condi-IP}.
Then, 
there exists a constant $\gamma >0$ for any $\varphi \in \mathscr{B}_{\gamma}$
and directly Riemann integrable function $\psi$ with compact support on $\bb R$, we have, as $n \to \infty$, 
uniformly in $|t| = o(\sqrt{n})$, $x = \bb R v$, $v \in V$ and $f \in V^*$ with $\|v\| = \|f\| =1$, 
\begin{align}\label{LLT-Moderate-01}
& \mathbb{E}  \Big[ \varphi(G_n \!\cdot\! x)
 \psi \Big( \log| \langle f, G_n v \rangle | - n\lambda_1 - \sqrt{n}\sigma t \Big) 
    \Big]   
     = \frac{e^{ -\frac{t^2}{2} + \frac{t^3}{\sqrt{n}}
  \zeta(\frac{t}{\sqrt{n}} ) }}{\sigma \sqrt{2 \pi n}}  
  \left[ \nu(\varphi) \int_{\bb R} \psi(u) du  +  o(1) \right].
\end{align}
\end{theorem}

In particular taking $\varphi=1$  and $\psi = \mathds 1_{[a_1,a_2]}$, where $a_1 < a_2$,
we get the asymptotic \eqref{LLT-MD-Intro}.
The asymptotic 
\eqref{LLT-Moderate-01} improves the recent results obtained in \cite{GQX20} and \cite{DKW21}: 
the result in \cite{GQX20} corresponds to the case when $t=0$, $\varphi=1$  and $\psi = \mathds 1_{[a_1,a_2]}$, 
and that in \cite{DKW21} to the case when $t = o(1)$, $\varphi=1$  and $\psi = \mathds 1_{[a_1,a_2]}$.

\section{Proof of Cram\'er type moderate deviation expansion} 

\subsection{Spectral gap properties and a change of measure}\label{subsec-Pz}
Denote by $B_{s_0}(0): = \{ z \in \bb{C}: |z| < s_0 \}$ 
the open disc with center $0$ and radius $s_0 >0$ in the complex plane $\bb C$. 
For any $z \in \bb{C}$ with $|\Re z| < s_0$ small enough
and any bounded measurable function $\varphi$ on $\bb P(V)$, 
consider 
\begin{align}\label{Def_Pz_Ch7}
P_z \varphi(x) = \int_{\textup{GL}(V)} e^{z \sigma(g, x)} \varphi(g \!\cdot\! x) \mu(dg),  
\quad  x \in \bb P(V). 
\end{align}
The following result  
shows that complex transfer operator $P_z$ has spectral gap properties.

\begin{lemma}[\cite{BQ16b, XGL19b}]  \label{Ch7transfer operator}
Assume \ref{Ch7Condi-Moment} and \ref{Ch7Condi-IP}.  
Then, there exist constants $\gamma>0$ and $s_0 >0$ such that for any $z \in B_{s_0}(0)$ and $n \geq 1$, 
\begin{align}\label{Ch7Pzn-decom}
P_z^n = \kappa^n(z) \nu_z \otimes r_z + L_z^n, 
\end{align}
where 
\begin{align*}
z \mapsto \kappa(z) \in \bb{C}, \quad   z \mapsto r_z \in \mathscr{B}_{\gamma} , 
\quad   z \mapsto \nu_z \in \mathscr{B}_{\gamma}' , 
\quad   z \mapsto  L_z \in \mathscr{L(B_{\gamma},B_{\gamma})}
\end{align*}
are analytic mappings which satisfy, for any $z \in B_{s_0}(0)$, 

\begin{itemize}
\item[{\rm(a)}]
    the operator $M_z: = \nu_z \otimes r_z$ is a rank one projection on $\mathscr{B}_{\gamma}$,
    i.e. $M_z \varphi = \nu_z(\varphi) r_z$ for any $\varphi \in \mathscr{B}_{\gamma}$; 

\item[{\rm(b)}]
 $M_z L_z = L_z M_z =0$,  $P_z r_z = \kappa(z) r_z$ with $\nu(r_z) = 1$, and  $\nu_z P_z = \kappa(z) \nu_z$;

\item[{\rm(c)}]
    $\kappa(0) = 1$, $r_0 = 1$, $\nu_0 = \nu$ with $\nu$ defined by \eqref{Ch7mu station meas}, and 
    $\kappa(z)$ and $r_z$ are strictly positive for real-valued $z \in (-s_0, s_0)$.    
\end{itemize}
\end{lemma}

Lemma \ref{Ch7transfer operator} allows us to obtain
a change of measure formula, under conditions \ref{Ch7Condi-Moment} and \ref{Ch7Condi-IP}. 
For any $x \in \bb P(V)$, $g \in \textup{GL}(V)$ and $s \in (-s_0, s_0)$ with $s_0>0$ sufficiently small, 
let
\begin{align*}
q_n^s(x, g) = \frac{ e^{s \sigma(g, x) } }{ \kappa^{n}(s) } \frac{ r_s(g \!\cdot\! x) }{ r_s(x) },
\quad  n \geq 1. 
\end{align*}
Using $P_s r_s = \kappa(s) r_s$ and the fact that $\kappa(s)$ and $r_s$ are strictly positive for $s \in (-s_0, s_0)$, 
we get that the probability measures
\begin{align*}
\bb Q_{s,n}^x (dg_1, \ldots, dg_n) = q_n^s(x, G_n) \mu(dg_1) \ldots \mu(dg_n),  \quad  n \geq 1, 
\end{align*}
form a projective system on $\textup{GL}(V)^{\bb{N}}$. 
Therefore, 
by the Kolmogorov extension theorem, 
there is a unique probability measure  $\bb Q_s^x$ on $\textup{GL}(V)^{\bb{N}}$ with marginals $\bb Q_{s,n}^x$. 
We write $\bb{E}_{\bb Q_s^x}$ for the corresponding expectation 
and the change of measure formula holds: 
for any $s \in (-s_0, s_0)$, $x \in \bb P(V)$, 
$n\geq 1$ and bounded measurable function $h$, 
\begin{align}\label{Ch7basic equ1}
&  \frac{1}{ \kappa^{n}(s) r_{s}(x) } \bb{E}  \Big[   
r_{s}(G_n \!\cdot\! x) e^{s \sigma(G_n, x) } h \Big( G_1 \!\cdot\! x, \sigma(G_1, x), \dots, G_n \!\cdot\! x, \sigma(G_n, x) \Big)
  \Big]   \nonumber\\
&  =   \bb{E}_{\bb{Q}_{s}^{x}} \Big[ h \Big( G_1 \!\cdot\! x, \sigma(G_1, x), \dots, G_n \!\cdot\! x, \sigma(G_n, x) \Big) \Big].
\end{align}
Under the changed measure $\bb Q_s^x$, the process $(G_n \!\cdot\! x)_{n \geq 0}$ is a Markov chain 
with the transition operator $Q_s$ given as follows: for any $\varphi \in \mathscr{C}(\bb P(V))$, 
\begin{align*}
Q_{s}\varphi(x) = \frac{1}{\kappa(s)r_{s}(x)}P_s(\varphi r_{s})(x),  \quad   x \in \bb P(V).
\end{align*}
By \cite{XGL19b}, 
the Markov operator $Q_s$ has a unique invariant probability measure $\pi_s$ given by 
\begin{align}\label{ExpCon-Qs}
\pi_s(\varphi) = \frac{ \nu_s( \varphi r_s) }{ \nu_s(r_s) }  \quad  \mbox{for any } \varphi \in \mathscr{C}(\bb P(V)).
\end{align}
We shall need the following property on the changed measure $\bb Q_s^x$. 

\begin{lemma}[\cite{XGL22a}]  \label{Lem_Regu_pi_s}
Assume \ref{Ch7Condi-Moment} and \ref{Ch7Condi-IP}. 
Then, for any $\ee >0$, there exist constants $s_0 >0$ and $c, C >0$ such that
for all $s \in (-s_0, s_0)$, $n \geq k \geq 1$, $x \in \bb P(V)$ and $y \in \bb P(V^*)$, 
\begin{align}\label{Regu_pi_s}
\bb Q_s^x \Big( \log \delta(y, G_n \!\cdot\! x) \leq -\ee k  \Big) \leq C e^{- ck}. 
\end{align}
\end{lemma} 

Note that \eqref{Regu_pi_s} implies the H\"{o}lder regularity of the invariant measure $\pi_s$: 
there exist constants $s_0 >0$ and $\eta > 0$ such that
\begin{align*} 
\sup_{ s\in (-s_0, s_0) } \sup_{y \in \bb P(V^*) } 
  \int_{\bb P(V) } \frac{1}{ \delta(y, x)^{\eta} } \pi_s(dx) < + \infty. 
\end{align*}

It is shown in \cite{XGL19b} that the strong law of large numbers 
for the norm cocycle $\sigma(G_n, x)$ under the measure $\bb Q_s^x$ holds: 
for any $s\in (-s_0, s_0)$ and $x \in \bb P(V)$, 
\begin{align*}
\lim_{n \to \infty} \frac{ \sigma(G_n, x) }{n} = \Lambda'(s),  \quad  \bb Q_s^x\mbox{-a.s.}
\end{align*}
where $\Lambda(s) = \log \kappa(s)$. 
For any $s\in (-s_0, s_0)$, $u \in \bb R$ and $\varphi \in \mathscr{C}(\bb P(V))$, 
define 
\begin{align} \label{Def_Rsz_Ch7}
R_{s, iu} \varphi(x)
= \bb{E}_{\bb{Q}_{s}^{x}} \left[ e^{iu( \sigma(g, x) - \Lambda'(s) )} \varphi(g \!\cdot\! x) \right],  
\quad   x \in \bb P(V). 
\end{align}
From the cocycle property of $\sigma(\cdot, \cdot)$, it follows that for any $n \geq 1$, 
\begin{align}\label{Def_Rsz-nth}
R^{n}_{s, iu} \varphi(x)
= \bb{E}_{\bb{Q}_{s}^{x}} \left[ e^{iu( \sigma(G_n, x) - n\Lambda'(s) )} \varphi(G_n \!\cdot\! x) \right],
\quad    x \in \bb P(V).
\end{align}
Now we give the spectral gap properties of the perturbed operator $R_{s, iu}$.

\begin{lemma}[\cite{XGL19b}] \label{Ch7perturbation thm}
Assume \ref{Ch7Condi-Moment} and \ref{Ch7Condi-IP}.  
Then, there exist constants $\gamma>0$, $s_0 >0$ and $\delta > 0$ 
such that for any $s \in (-s_0, s_0)$ and $u \in  (-\delta, \delta)$,
\begin{align} \label{Ch7perturb001}
R^{n}_{s, iu} = \lambda^{n}_{s, iu} \Pi_{s, iu} + N^{n}_{s, iu},  
\end{align}
where
\begin{align}\label{relationlamkappa001}
\lambda_{s, iu} = e^{ \Lambda(s + iu) - \Lambda(s) - iu \Lambda'(s)}, 
\end{align}
and for fixed $s \in (-s_0, s_0)$, the mappings $u \mapsto \Pi_{s, iu}: (-\delta, \delta) \to \mathscr{L(B_{\gamma},B_{\gamma})}$,
$u \mapsto N_{s, iu}: (-\delta, \delta) \to \mathscr{L(B_{\gamma},B_{\gamma})}$
and $u \mapsto \lambda_{s, iu}: (-\delta, \delta) \to \bb{R}$ are analytic. 
In addition, for fixed $s$ and $u$, the operator $\Pi_{s,iu}$ is a rank-one projection with 
$\Pi_{s, 0}(\varphi)(x) = \pi_{s}(\varphi)$ for any $\varphi \in \mathscr{B}_{\gamma}$ and $x\in \bb P(V)$,
and $\Pi_{s,iu} N_{s, iu} = N_{s,iu} \Pi_{s,iu} = 0$.  
Moreover,  for any $k \in  \bb{N}$,  there exist constants $c > 0$ and $0< a <1$ such that
for any $s \in (-s_0, s_0)$ and $u \in  (-\delta, \delta)$, 
\begin{align} \label{Ch7SpGapContrN}
\Big\| \frac{d^{k}}{du^{k}} \Pi^{n}_{s, iu}  \Big\|_{\mathscr{B}_{\gamma} \to \mathscr{B}_{\gamma}} 
\leq c,  
\quad
\Big\| \frac{d^{k}}{du^{k}}N^{n}_{s, iu}  \Big\|_{\mathscr{B}_{\gamma} \to \mathscr{B}_{\gamma}} 
\leq  c  a^n.    
\end{align}
\end{lemma} 

Now we give the non-arithmetic property of the perturbed operator $R_{s,iu}$. 

\begin{lemma}[\cite{XGL19b}] \label{Lem-St-NonLatt}
Assume \ref{Ch7Condi-Moment} and \ref{Ch7Condi-IP}.  
Then, for any compact set $K \subseteq \mathbb{R}\backslash\{0\}$,
there exist constants $\gamma, s_0, c_1, c_2>0$ such that 
for any $n \geq 1$ and $\varphi\in \mathscr{B}_{\gamma}$,
\begin{align*} 
\sup_{s \in (-s_0, s_0)} \sup_{u\in K} \sup_{x\in \bb P(V)}
|R^{n}_{s, iu}\varphi(x)| \leq  c_1 e^{- c_2 n} \|\varphi \|_{\gamma}.
\end{align*}
\end{lemma}

\subsection{Smoothing inequality} \label{Ch7secAuxres001}

The Fourier transform of an integrable function $h: \bb{R} \to \bb{C}$ is defined by 
$\widehat{h}(u) = \int_{\bb{R}} e^{-iuw} h(w) dw,$  $u \in \bb{R}$. 
If $\widehat{h}$ is integrable on $\bb{R}$, then by the Fourier inversion formula we have 
$h(w) = \frac{1}{2 \pi} \int_{\bb{R}} e^{iuw} \widehat{h}(u) du,$ for almost all $w \in \bb{R}$
with respect to the Lebesgue measure on $\bb R$.

Now we fix a non-negative density function $\rho$ on $\bb{R}$ satisfying $\rho(w) \leq \frac{c}{w^4}$ for $w \geq 1$,  
whose Fourier transform $\widehat{\rho}$ is a non-negative Lipschitz continuous function with support on $[-1, 1]$. 
The existence of such a function is shown in \cite{GX21}. 
For any $0< \ee <1$, define the scaled density function
$\rho_{\ee}(w) = \frac{1}{\ee} \rho(\frac{w}{\ee})$, $w \in \bb R,$ 
whose Fourier transform $\widehat{\rho}_{\ee}$ has support on $[-\ee^{-1},\ee^{-1}]$.

For any $\ee >0$ and non-negative integrable function $\psi$ on $\bb R$, set
\begin{align}\label{smoo001}
\psi^+_{\varepsilon}(w) = \sup_{|w - w'| \leq \ee} \psi(w') 
\quad  \mbox{and}  \quad 
\psi^-_{\varepsilon}(w) = \inf_{|w - w'| \leq \ee} \psi(w'),
\quad w \in \bb R. 
\end{align}
We need the following smoothing inequality which is shown in \cite{GLL17}.
Denote by $h_1 * h_2$ the convolution of functions $h_1$ and $h_2$ on the real line. 

\begin{lemma}[\cite{GLL17}]  \label{estimate u convo}
Assume that  $\psi$ is a non-negative integrable function on $\bb R$ and that 
${\psi}^+_{\varepsilon}$ and ${\psi}^-_{\varepsilon}$ are measurable for any $\varepsilon \in (0,1)$. 
Then, there exists a positive constant $c_{\rho}(\ee)$ with $c_{\rho}(\ee) \to 0$ as $\ee \to 0$,
such that for any $w \in \mathbb{R}$, 
\begin{align}
\psi^-_{\ee} * \rho_{\ee^2}(w) - 
\int_{ |u| \geq \ee } \psi^-_{\varepsilon}( w-u ) \rho_{\varepsilon^2}(u) du
\leq \psi(w) \leq (1+ c_{\rho}(\ee))
\psi^+_{\varepsilon} * \rho_{\varepsilon^2}(w). \nonumber
\end{align}
\end{lemma}

\subsection{Asymptotic expansions of the perturbed operator}
The goal of this section is to 
establish the precise asymptotics of 
the perturbed operator $R_{s, it}$, 
which will play an important role for establishing 
the Cram\'{e}r type moderate deviation expansion for the coefficients $\langle f, G_n v \rangle$
in Theorem \ref{Thm-Cram-Entry_bb}. 
In the sequel, for any fixed $t >1$, 
we shall choose $s>0$  satisfying the following equation:
\begin{align}\label{SaddleEqua}
\Lambda'(s) - \Lambda'(0) = \frac{\sigma t}{\sqrt{n}}. 
\end{align}
For brevity we denote $\sigma_s = \sqrt{\Lambda''(s)}$. 
By \cite{XGL19b}, the function $\Lambda$ is strictly convex in a small neighborhood of $0$, 
so that $\sigma_s >0$ uniformly in $s \in (0, s_0)$.

For $s >0$, let 
\begin{align*}
\psi_s(w) = e^{-sw} \mathds{1}_{\{ w \geq 0 \} },  \quad  w \in \bb{R}. 
\end{align*}
With the notation $\psi_{s,\ee}^{-}(w) = \inf_{|w - w'| \leq \ee} \psi_s(w')$ (cf. \eqref{smoo001}),
we have that for any $\ee \in (0, 1)$, 
\begin{align}\label{Def-psi-see}
\psi_{s,\ee}^{-}(w) = e^{-s(w + \ee)} \mathds{1}_{ \{w \geq \ee\} },  \quad  w  \in \bb{R}. 
\end{align}
For $s >0$ and $\ee \in (0,1)$, 
the Fourier transform of $\psi_{s,\ee}^{-}$ is given by 
\begin{align}\label{Diffepsi}
\widehat {\psi}^-_{s,\varepsilon}(u) 
= \int_{\bb R} e^{-iuw} \psi_{s,\ee}^{-}(w) dw
=  e^{-2 \varepsilon s} \frac{e^{- i \varepsilon u} }{s + iu},  \quad u \in \bb R. 
\end{align}

The following proposition will be used to 
establish the upper tail Cram\'er type moderate deviation expansion \eqref{LD general upper001}. 
Note that the explicit dependence on $\ee,  t,  l$ and $\varphi$ will play a crucial role.  

\begin{proposition}\label{KeyPropo}
Assume \ref{Ch7Condi-Moment} and \ref{Ch7Condi-IP}.
Let $s > 0$ be such that \eqref{SaddleEqua} holds. 
Let $(t_n)_{n \geq 1}$ be a sequence of positive numbers such that $t_n \to \infty$ and $\frac{t_n}{\sqrt{n}} \to 0$ as $n \to \infty$. 
Then, there exist constants $\gamma, s_0, c, c_{\ee} > 0$ such that 
for any $\ee \in (0,1)$, $s \in (0, s_0)$, $x\in \bb P(V)$, $t \in [t_n, o(\sqrt{n} )]$, 
$\varphi \in \mathscr{B}_{\gamma}$ and $l \in \bb R$, 
\begin{align} \label{KeyPropUpper}
&   \left|  s \sigma_s \sqrt{n}  \int_{\mathbb{R}} e^{-iuln} R_{s, iu}^n (\varphi)(x) 
\widehat {\psi}^-_{s,\varepsilon}(u) \widehat\rho_{\varepsilon^{2}}(u) du 
 -  \sqrt{ 2 \pi }  \pi_s(\varphi)  \right|  \nonumber\\
&  \leq  \left( \frac{ct}{\sqrt{n}}  +  \frac{c_{\ee}}{\sqrt{n}} +  \frac{c}{t^2}  \right)  \| \varphi \|_{\infty}
     + c \left( |l| \sqrt{n} + \frac{1}{\sqrt{n}}  +  e^{-c_{\ee} n } \right) \|\varphi\|_{\gamma}.  
\end{align}   
\end{proposition}

\begin{proof}
Without loss of generality, we assume that the target function $\varphi$ is non-negative on $\bb P(V)$. 
By Lemma \ref{Ch7perturbation thm}, we have the following decomposition:
with $\delta > 0$ small enough,
\begin{align} \label{Thm1 integral1 J Nol}
s \sigma_s \sqrt{n} \int_{\mathbb{R}} e^{-iuln}  R_{s, iu}^n (\varphi)(x) 
\widehat {\psi}^-_{s,\varepsilon}(u) \widehat\rho_{\varepsilon^{2}}(u) du  
 =: I_1 + I_2 + I_3, 
\end{align}
where 
\begin{align*} 
I_1  & =  s \sigma_s \sqrt{n} \int_{|u|\geq\delta}  e^{-iuln}  R^{n}_{s, iu}(\varphi)(x) 
              \widehat {\psi}^-_{s,\varepsilon}(u) \widehat\rho_{\varepsilon^{2}}(u) du,  \nonumber\\
I_2  & =  s \sigma_s \sqrt{n}  \int_{|u|< \delta }  e^{-iuln}  N^{n}_{s, iu}(\varphi )(x) 
               \widehat {\psi}^-_{s,\varepsilon}(u) \widehat\rho_{\varepsilon^{2}}(u) du,   \nonumber\\
I_3  & =  s \sigma_s \sqrt{n}  \int_{|u|<\delta}  e^{-iuln}  \lambda^{n}_{s, iu} \Pi_{s, iu}(\varphi )(x) 
               \widehat {\psi}^-_{s,\varepsilon}(u) \widehat\rho_{\varepsilon^{2}}(u) du.        
\end{align*}
For simplicity, we denote $K_s(iu)= \log \lambda_{s, iu}$ and choose the branch such that $K_s(0)=0$.
Using \eqref{relationlamkappa001}, we have that for $u \in (-\delta, \delta)$, 
\begin{align}\label{formula-Ks-iu}
K_s(iu) = \Lambda(s + iu) - \Lambda(s) - iu \Lambda'(s). 
\end{align}
Since the function $\Lambda$ is analytic in a small neighborhood of $0$,
using Taylor's formula yields that for $u \in (-\delta, \delta)$, 
\begin{align} \label{Ch7Expan Ks 01}
K_s(iu) = \sum_{k=2}^{\infty}  \frac{\Lambda^{(k)}(s) }{k!} (iu)^k,
\quad \mbox{where}  \  \Lambda(s) = \log \kappa(s) 
\end{align}
and 
\begin{align} \label{Expan Ks 02}
\Lambda'(s)-\Lambda'(0)=\sum_{k=2}^{\infty}\frac{\gamma_k}{(k-1)!}s^{k-1},  \quad
\mbox{where} \   \gamma_k = \Lambda^{(k)}(0).   
\end{align}
From \eref{SaddleEqua} and \eref{Expan Ks 02}, we see that
\begin{align}\label{equation t and s 1}
 \frac{ \sigma t }{ \sqrt{n} } = \sum_{k=2}^{\infty} \frac{ \gamma_k }{ (k-1)! } s^{k-1}.
\end{align}
Since $\gamma_2 = \sigma^2>0$,  from \eqref{equation t and s 1} 
we deduce that for any $t >1$ and sufficiently large $n \geq 1$, 
the equation \eqref{SaddleEqua} has a unique solution given by
\begin{align}\label{root s 1}
s = \frac{1}{\gamma_2^{ 1/2 }} \frac{t}{\sqrt{n}} 
    - \frac{\gamma_3}{ 2 \gamma_2^2 } \left( \frac{t}{ \sqrt{n} }  \right)^2
    - \frac{\gamma_4 \gamma_2 - 3\gamma_3^2}{6\gamma_2^{7/2} } \left( \frac{t}{ \sqrt{n} } \right)^3
    + \cdots. 
\end{align}
For sufficiently large $n \geq 1$, the series on the right-hand side of \eqref{root s 1} is absolutely convergent
according to the theorem on the inversion of analytic functions. 
Besides, from \eqref{SaddleEqua} and $t = o(\sqrt{n})$ we see that $s \to 0^+$ as $n \to \infty$, 
so that we can assume $s \in (0, s_0)$ for sufficiently small constant $s_0 >0$.

\textit{Estimate of $I_1$.} 
Since the function $\widehat\rho_{\varepsilon^{2}}$ is supported on $[-\ee^{-2}, \ee^{-2}]$, 
by Lemma \ref{Lem-St-NonLatt}, for fixed $\delta >0$, there exist constants $c_{\ee}, C_{\ee} > 0$ such that
for any $\varphi \in \mathscr{B}_{\gamma}$, 
\begin{align}\label{EstiRsit}
\sup_{s \in (0, s_0)} \sup_{\delta \leq |u| \leq \ee^{-2} }  
\sup_{x\in \bb P(V) } |R^{n}_{s, iu}\varphi(x)|  \leq   C_{\ee} e^{-c_{\ee} n} \|\varphi \|_{\gamma}.  
\end{align}
From \eqref{Diffepsi} and the fact that $\rho_{\ee^2}$ is a density function on $\bb R$, we see that 
\begin{align}\label{Contrpsirho}
\sup_{u \in \bb{R}} |\widehat {\psi}^-_{s,\ee}(u)|
\leq \widehat {\psi}^-_{s,\ee}(0) = \frac{1}{s} e^{ -2\ee s } \leq \frac{1}{s}, 
\quad  
\sup_{u \in \bb{R}} |\widehat\rho_{\ee^{2}}(u)| 
\leq \widehat\rho_{\ee^{2}}(0) = 1.  
\end{align}
Using \eqref{EstiRsit} and the first inequality in \eqref{Contrpsirho}, 
and taking into account that the function $\widehat\rho_{\ee^{2}}$ is integrable on $\bb{R}$, 
we obtain the desired bound for $I_1$: 
for fixed $\delta >0$, there exist constants  $c_{\ee}, C_{\ee} > 0$ such that
for any $s \in (0, s_0)$, $l \in \bb R$, $x \in \bb P(V)$ and $\varphi \in \mathscr{B}_{\gamma}$, 
\begin{align}\label{EstimI1n00a}
|I_1| \leq   C_{\ee} e^{-c_{\ee} n } \|\varphi \|_{\gamma}. 
\end{align}

\textit{Estimate of $I_2$.}
Using \eqref{Ch7SpGapContrN}, we have
that uniformly in $s \in (0, s_0)$, $u \in (-\delta, \delta)$, $x \in \bb P(V)$ and $\varphi \in \mathscr{B}_{\gamma}$, 
\begin{align*}
| N^{n}_{s, iu}(\varphi )(x) |
\leq \| N^{n}_{s, iu} \|_{\mathscr{B}_{\gamma} \to \mathscr{B}_{\gamma}} \|\varphi \|_{ \gamma}
\leq  C e^{-cn} \|\varphi \|_{ \gamma}.
\end{align*} 
This, together with \eqref{Contrpsirho}, implies the desired bound for $I_2$: 
for fixed small $\delta >0$, there exist constants $c, c_{\ee} > 0$ such that
for any $s \in (0, s_0)$, $l \in \bb R$, $x \in \bb P(V)$ and $\varphi \in \mathscr{B}_{\gamma}$, 
\begin{align}\label{Esti I2n}
 |I_2|  \leq  C e^{-cn}  \|\varphi \|_{\gamma}. 
\end{align}

\textit{Estimate of $I_3$.}
For brevity, we denote for $s \in (0, s_0)$ and $x \in \bb P(V)$,
\begin{align}\label{Defpsisxt}
\Psi_{s,x}(u):= \Pi_{s, iu}(\varphi)(x) 
\widehat {\psi}^-_{s,\varepsilon}(u) \widehat\rho_{\varepsilon^{2}}(u),  \quad  -\delta < u < \delta. 
\end{align}
Recalling that $K_s(iu)= \log \lambda_{s,iu}$, we decompose the term $I_3$ into two parts: 
\begin{align}  \label{DecompoJ1}
I_3 = s \sigma_s \sqrt{n}  \int_{ |u| < \delta  }  e^{n K_s( iu) - iu ln} \Psi_{s,x}(u) du
   = I_{31} + I_{32}, 
\end{align}
where 
\begin{align*}
&   I_{31} = s \sigma_s \sqrt{n}  
  \int_{ n^{- \frac{1}{2}} \log n \leq |u| < \delta  }  e^{n K_s( iu) - iu ln} \Psi_{s,x}(u) du,  \nonumber\\
&   I_{32}  = s \sigma_s \sqrt{n}  \int_{ |u| <  n^{-\frac{1}{2}} \log n }  e^{n K_s( iu) - iu ln} \Psi_{s,x}(u) du. 
\end{align*}

\textit{Estimate of $I_{31}$.}
By \eqref{Ch7SpGapContrN}, there exists a constant $c >0$ 
such that for any $s \in (0, s_0)$, $|u| < \delta$, $x \in \bb P(V)$ and $\varphi \in \mathscr{B}_{\gamma}$, 
\begin{align}\label{ContrPirs}
|\Pi_{s, iu}(\varphi)(x)| \leq c \|\varphi\|_{\gamma}.
\end{align}
This, together with \eqref{Contrpsirho}, yields that
there exists a constant $c >0$ such that 
for any $s \in (0, s_0)$, $|u| < \delta$, $x \in \bb P(V)$ and $\varphi \in \mathscr{B}_{\gamma}$,  
\begin{align} \label{Estipsi01}
|\Psi_{s,x}(u)| \leq \frac{1}{s} |\Pi_{s, iu}(\varphi)(x)| \leq  \frac{c}{s} \|\varphi\|_{\gamma}.
\end{align}
Using \eqref{Ch7Expan Ks 01} and noting that $\Lambda''(s) = \sigma_s^2>0 $, 
we find that there exists a constant $c>0$ such that for any $s \in (0, s_0)$ and $u \in (-\delta, \delta)$, 
\begin{align*}
\Re \big( K_s(iu) \big) 
= \Re \left( \sum_{k=2}^{\infty}  \frac{\Lambda^{(k)}(s) }{k!} (iu)^k \right) 
< -\frac{1}{4} \sigma_s^2 u^2 < - c u^2. 
\end{align*}
Combining this with \eqref{Estipsi01}, we derive that 
there exists a constant $c >0$ such that for any $s \in (0, s_0)$, $l \in \bb R$, $x \in \bb P(V)$ and $\varphi \in \mathscr{B}_{\gamma}$,  
\begin{align} \label{EstiJ11}
|I_{31}| 
& \leq  c \sqrt{n}  \|\varphi\|_{\gamma}  \int_{ n^{- \frac{1}{2}} \log n \leq |u| < \delta  }  \left| e^{n K_s( iu) } \right| du   \notag\\ 
& \leq  c \sqrt{n}  \|\varphi\|_{\gamma}  \int_{ n^{ - \frac{1}{2} }  \log n  \leq |u| <  \delta  } e^{- c n u^2 }  du  \notag\\
& \leq  c   \|\varphi\|_{\gamma}  \int_{   \log n  \leq |w| <  \delta \sqrt{n} } e^{- c w^2 }  dw  \notag\\
& \leq \frac{c}{ \sqrt{n} }  \|\varphi\|_{\gamma}. 
\end{align}

\textit{Estimate of $I_{32}$.}
By a change of variable $u' = \sigma_s \sqrt{n} u$ and using \eqref{Ch7Expan Ks 01}, we get
\begin{align}\label{I32-decomposition-yy}
 I_{32}  
 &  =  s \int_{- \sigma_s \log n}^{\sigma_s \log n }
  e^{n K_s( \frac{iu}{ \sigma_s \sqrt{n} } ) }
  e^{-i \frac{l \sqrt{n}}{\sigma_s}  u } \Psi_{s,x} \left( \frac{u}{\sigma_s \sqrt{n}} \right) du  \nonumber\\
 &  =  s \int_{- \sigma_s \log n}^{\sigma_s \log n }
  e^{-\frac{u^2}{2}}  e^{n K_s( \frac{iu}{ \sigma_s \sqrt{n} } ) + \frac{u^2}{2} }   
  e^{-i \frac{l \sqrt{n}}{\sigma_s}  u } \Psi_{s,x} \left( \frac{u}{\sigma_s \sqrt{n}} \right) du  \nonumber\\
& = I_{321} + I_{322} + I_{323}, 
\end{align}
where 
\begin{align*}
&  I_{321} = s \int_{- \sigma_s \log n}^{\sigma_s \log n }  e^{-\frac{u^2}{2}}  
  \left[ e^{n K_s( \frac{iu}{ \sigma_s \sqrt{n} } ) + \frac{u^2}{2} }   - 1 \right] 
  e^{-i \frac{l \sqrt{n}}{\sigma_s}  u }   \Psi_{s,x} \left( \frac{u}{\sigma_s \sqrt{n}} \right) du  \nonumber\\ 
&  I_{322} = s \int_{- \sigma_s \log n}^{\sigma_s \log n }  e^{-\frac{u^2}{2}}   
  \left[ e^{-i \frac{l \sqrt{n}}{\sigma_s}  u } - 1 \right] 
   \Psi_{s,x} \left( \frac{u}{\sigma_s \sqrt{n}} \right) du  \nonumber\\ 
&  I_{323} = s \int_{- \sigma_s \log n}^{\sigma_s \log n }  e^{-\frac{u^2}{2}}    
   \Psi_{s,x} \left( \frac{u}{\sigma_s \sqrt{n}} \right) du. 
\end{align*}

\textit{Estimate of $I_{321}$.}
Using \eqref{formula-Ks-iu} and Taylor's expansion, we get that for any $|u| \leq \sigma_s \log n$ and $s \in (0, s_0)$, 
\begin{align*}
\left| n K_s \left( \frac{iu}{ \sigma_s \sqrt{n} } \right) + \frac{u^2}{2} \right| \leq  c \frac{|u|^3}{\sqrt{n}}. 
\end{align*}
By the inequality $| e^{z}-1 | \leq |z|e^{|z|}$ for $z \in \bb C$, 
it follows that for any $|u| \leq \sigma_s \log n$ and $s \in (0, s_0)$, 
\begin{align*} 
& \left| e^{n K_s( \frac{iu}{ \sigma_s \sqrt{n} } ) + \frac{u^2}{2} }   - 1 \right|
\leq  c \frac{|u|^3}{\sqrt{n}}  e^{ c \frac{|u|^3}{\sqrt{n}} } 
\leq c' \frac{|u|^3}{\sqrt{n}}. 
\end{align*}
%
%
From this, using \eqref{Estipsi01} and the fact that $|e^{-iul \sqrt{n} / \sigma_s}| = 1$, we obtain that
there exists a constant $c >0$ such that for any $s \in (0, s_0)$, $l \in \bb R$, $x \in \bb P(V)$ and $\varphi \in \mathscr{B}_{\gamma}$, 
\begin{align}\label{Esti_I321_aa}
|I_{321}| 
& \leq  \frac{cs}{\sqrt{n}} \int_{- \sigma_s \log n}^{\sigma_s \log n }  
      e^{-\frac{u^2}{2}}  |u|^3  \left| \Psi_{s,x} \left( \frac{u}{\sigma_s \sqrt{n}} \right) \right|  du    \notag\\
& \leq  \frac{c}{\sqrt{n}}  \|\varphi\|_{\gamma}  \int_{- \sigma_s \log n}^{\sigma_s \log n }  e^{-\frac{u^2}{2}}  |u|^3  du
\leq \frac{c}{\sqrt{n}} \|\varphi\|_{\gamma}. 
\end{align}

\textit{Estimate of $I_{322}$.}
Since $|e^{-iul \sqrt{n} / \sigma_s} - 1| \leq |u l \sqrt{n} / \sigma_s| \leq c |ul| \sqrt{n}$, 
using again \eqref{Estipsi01}, we get that 
there exists a constant $c >0$ such that for any $s \in (0, s_0)$, $l \in \bb R$, $x \in \bb P(V)$ and $\varphi \in \mathscr{B}_{\gamma}$, 
\begin{align}\label{Esti_I322_aa}
|I_{322}| 
\leq   cs |l| \sqrt{n} \int_{- \sigma_s \log n}^{\sigma_s \log n }  
      e^{-\frac{u^2}{2}}  |u|  \left| \Psi_{s,x} \left( \frac{u}{\sigma_s \sqrt{n}} \right) \right|  du
\leq  c |l| \sqrt{n} \|\varphi\|_{\gamma}. 
\end{align}

\textit{Estimate of $I_{323}$.}
We shall establish the following bound for $I_{323}$: there exist constants $c, c_{\ee} > 0$ such that
for any $s \in (0, s_0)$, $x\in \bb P(V)$, $t \in [t_n, o(\sqrt{n} )]$
and $\varphi \in \mathscr{B}_{\gamma}$,
\begin{align}\label{Esti_I32_bb}
\left|  I_{323} -   \sqrt{ 2 \pi } \pi_s(\varphi)  \right| 
\leq   \left( \frac{ct}{\sqrt{n}}  +  \frac{c_{\ee}}{\sqrt{n}} +  \frac{c}{t^2}  \right)  \| \varphi \|_{\infty}
   +  \frac{c}{\sqrt{n}} \| \varphi \|_{\gamma}. 
\end{align}
To prove \eqref{Esti_I32_bb}, 
we denote 
\begin{align}\label{def-un}
u_n = \frac{u}{ \sigma_s \sqrt{n} }, 
\end{align}
and, in view of \eqref{Defpsisxt}, we write 
\begin{align*}
\Psi_{s,x} (u_n) =  h_1(u_n) + h_2(u_n) + h_3(u_n) + h_4(u_n), 
\end{align*}
where
\begin{align*}
&   h_1(u_n) =  \big[ \Pi_{s, i u_n}(\varphi)(x) -  \pi_s(\varphi) \big] 
\widehat {\psi}^-_{s,\varepsilon}(u_n)
\widehat\rho_{\varepsilon^{2}}(u_n),    \nonumber\\
&   h_2(u_n)  =    \pi_s(\varphi) \widehat {\psi}^-_{s,\varepsilon}(u_n) 
\left[ \widehat\rho_{\varepsilon^{2}}(u_n) -  \widehat\rho_{\varepsilon^{2}}(0)  \right]  \nonumber\\
&   h_3(u_n) =   \pi_s(\varphi) 
 \left[ \widehat {\psi}^-_{s,\varepsilon}(u_n) - \widehat {\psi}^-_{s,\varepsilon}(0) \right]
   \widehat\rho_{\varepsilon^{2}}(0),  \nonumber\\
&   h_4(u_n) =   \pi_s(\varphi) \widehat {\psi}^-_{s,\varepsilon}(0)
\widehat\rho_{\varepsilon^{2}}(0). 
\end{align*}
With the above notation, the term $I_{323}$ can be decomposed into four parts: 
\begin{align}\label{DecompoJ122}
I_{323} = J_{1} + J_{2} + J_{3} + J_{4}, 
\end{align}
where for $j = 1,2,3,4$, 
\begin{align*}
J_{j} =  s  \int_{- \sigma_s \log n }^{ \sigma_s \log n } e^{-\frac{u^2}{2}} 
h_j(u_n) du. 
\end{align*}

\textit{Estimate of $J_{1}$.}
By \eqref{Ch7SpGapContrN} and \eqref{def-un}, we have
$|\Pi_{s, i u_n}(\varphi)(x) -  \pi_s(\varphi)|
\leq  c  \frac{|u|}{\sqrt{n}} \| \varphi \|_{\gamma}$, uniformly in $x \in \bb P(V)$, $s \in (0, s_0)$
and $|u| \leq \sigma_s \log n $. 
Combining this with \eqref{Contrpsirho} gives
$|h_1(u_n)|  \leq  c  \frac{|u|}{ s \sqrt{n}} \| \varphi \|_{\gamma},$ 
and hence
\begin{align}\label{Controlh1}
| J_1 | \leq  \frac{c}{ \sqrt{n} } \| \varphi \|_{\gamma}. 
\end{align}

\textit{Estimate of $J_{2}$.}
Since $\widehat\rho_{\varepsilon^{2}}$ is Lipschitz continuous on $\bb R$, 
we have 
$|\widehat\rho_{\varepsilon^{2}}(u_n) -  \widehat\rho_{\varepsilon^{2}}(0)| \leq c \frac{|u_n|}{\varepsilon^4} $.
This, together with \eqref{Contrpsirho}, implies that  
$|h_2(u_n)|  \leq  \frac{c}{\varepsilon^4}  \frac{|u|}{ s \sqrt{n}} \|\varphi\|_{\infty},$
so that 
\begin{align}\label{Controlh2}
|J_2| \leq \frac{c}{\varepsilon^4}  \frac{1}{ \sqrt{n} } \|\varphi\|_{\infty}
 \leq  \frac{c_{\ee}}{ \sqrt{n} } \|\varphi\|_{\infty}.
\end{align}

\textit{Estimate of $J_{3}$.}
By the definition of the function $\widehat {\psi}^-_{s,\varepsilon}$ (see \eqref{Diffepsi}), we have 
\begin{align*}
 \widehat {\psi}^-_{s,\varepsilon}(u_n) - \widehat {\psi}^-_{s,\varepsilon}(0) 
& =  e^{-2 \varepsilon s} \left( \frac{e^{- i \varepsilon u_n} }{s + iu_n} - \frac{1}{s}  \right) \nonumber\\
& =  e^{-2 \varepsilon s}  e^{- i \varepsilon u_n} \left( \frac{1}{s + iu_n} - \frac{1}{s} \right)
    + e^{-2 \varepsilon s} \frac{1}{s} \left( e^{- i \varepsilon u_n} - 1 \right)   \nonumber\\
& =: A_1 (u)  +  A_2 (u) + A_3 (u), 
\end{align*}
where 
\begin{align*}
&  A_1 (u) =  e^{-2 \varepsilon s} \left( e^{- i \varepsilon u_n} - 1 \right) \frac{-i s u_n - u_n^2}{s(s^2 + u_n^2)},   \notag\\
& A_2 (u) =  e^{-2 \varepsilon s} \frac{-i s u_n - u_n^2}{s(s^2 + u_n^2)}, 
 \qquad  A_3 (u) =  e^{-2 \varepsilon s} \frac{1}{s} \left( e^{- i \varepsilon u_n} - 1 \right). 
\end{align*}
Since $\widehat\rho_{\varepsilon^{2}}(0) = 1$, it follows that $J_3 = J_{31} + J_{32} + J_{33}$, 
where  
\begin{align*}
J_{3j} = s \pi_s(\varphi) 
\int_{- \sigma_s \log n }^{ \sigma_s \log n } e^{-\frac{u^2}{2}}  A_j(u) du,  \quad  j = 1, 2, 3.  
\end{align*}
For $J_{31}$, 
using the inequality $|e^z - 1| \leq e^{\Re z} |z|$ for $z \in \bb C$,  \eqref{def-un} and \eqref{SaddleEqua}, we get 
\begin{align*}
|A_1 (u)|  & \leq  |u_n| \frac{s |u_n| + u_n^2}{s(s^2 + u_n^2)}  
 =   \left( 1 +  \frac{|u_n|}{s}  \right) \frac{u_n^2}{s^2 + u_n^2} 
 \leq c \left(  1 +  \frac{|u|}{t} \right) \frac{u^2}{t^2}.
 \end{align*}
 Since $|\pi_s (\varphi)| \leq c \|\varphi\|_{\infty}$ 
 and $s = O(\frac{t}{\sqrt{n}})$, $t \in [t_n, o(\sqrt{n} )]$ with $t_n \to \infty$ as $n \to \infty$, 
 there exists a constant $c>0$ such that for all $n \geq 1$, 
$s \in (0, s_0)$ and $\varphi \in \mathscr{B}_{\gamma}$, 
\begin{align}\label{Esti_J31}
|J_{31}| \leq \frac{c}{ \sqrt{n} } \|\varphi\|_{\infty}. 
\end{align}
For $J_{32}$, using the fact that the integral of an odd function 
over a symmetric interval is identically zero, 
by \eqref{def-un}, \eqref{SaddleEqua} and elementary calculations we deduce that there exists a constant $c>0$ such that for any $n \geq 1$, 
$s \in (0, s_0)$ and $\varphi \in \mathscr{B}_{\gamma}$, 
\begin{align}\label{Esti_J32}
|J_{32}| = e^{-2 \varepsilon s} \pi_s(\varphi) \int_{- \sigma_s \log n }^{ \sigma_s \log n } e^{-\frac{u^2}{2}}  
     \frac{u_n^2}{s^2 + u_n^2} du 
 \leq \frac{c}{t^2} \|\varphi\|_{\infty}. 
\end{align}
For $J_{33}$, using again \eqref{def-un} and the inequality $|e^z - 1| \leq e^{\Re z} |z|$ for $z \in \bb C$,
we see that there exists a constant $c>0$ such that for any $n \geq 1$, 
$s \in (0, s_0)$ and $\varphi \in \mathscr{B}_{\gamma}$, 
\begin{align}\label{Esti_J33}
|J_{33}| 
\leq  c \|\varphi\|_{\infty}  \int_{- \sigma_s \log n }^{ \sigma_s \log n } e^{-\frac{u^2}{2}}  \frac{|u|}{\sigma_s \sqrt{n}} du
\leq \frac{c}{ \sqrt{n} } \|\varphi\|_{\infty}. 
\end{align}
Consequently, putting together the bounds \eqref{Esti_J31}, \eqref{Esti_J32} and \eqref{Esti_J33}, we obtain
\begin{align} \label{IntegJ123Nol}
J_{3} \leq  \frac{c}{ \sqrt{n} } \|\varphi\|_{\infty} +   \frac{c}{ t^2 }  \|\varphi\|_{ \infty }. 
\end{align}

\textit{Estimate of $J_{4}$.}
It follows from \eqref{Diffepsi} and $\widehat\rho_{\varepsilon^{2}}(0) = 1$ that 
\begin{align}\label{EstiInteJ4n}
J_{4} 
 =  e^{ -2\varepsilon s } \pi_s(\varphi)    
  \int_{- \sigma_s \log n }^{ \sigma_s \log n }
e^{-\frac{u^2}{2}}  du. 
\end{align}
Since there exists a constant $c>0$ such that for any $s \in (0, s_0)$ and $n \geq 1$, 
\begin{align*} 
\sqrt{ 2 \pi }  >  \int_{-\sigma_s \log n  }^{ \sigma_s \log n } e^{- \frac{u^2}{2}}  du  
 >  \sqrt{ 2 \pi }  - \frac{c}{n}, 
\end{align*}
we get 
\begin{align*}
J_4 =  \sqrt{ 2 \pi }  e^{ -2\varepsilon s }  \pi_s(\varphi)   \left[ 1 + O \Big( \frac{1}{n} \Big) \right]. 
\end{align*}
Since $s = O(\frac{t}{\sqrt{n}})$ and $t >1$, it follows that
\begin{align}\label{EstiIntegral00a}
\Big|  J_4 -   \sqrt{ 2 \pi } \pi_s(\varphi)  \Big| \leq  \frac{ct}{\sqrt{n}} \| \varphi \|_{\infty}. 
\end{align} 
In view of \eqref{DecompoJ122}, putting together the bounds \eqref{Controlh1}, \eqref{Controlh2}, 
\eqref{IntegJ123Nol} and \eqref{EstiIntegral00a}, and noting that $\|\varphi\|_{\infty} \leq \|\varphi\|_{\gamma}$, 
we obtain the desired bound \eqref{Esti_I32_bb}. 
Putting together \eqref{Esti_I321_aa}, \eqref{Esti_I322_aa} and \eqref{Esti_I32_bb}, 
we derive that 
there exist constants $c, c_{\ee} > 0$ such that
for any $s \in (0, s_0)$, $x\in \bb P(V)$, $t \in [t_n, o(\sqrt{n} )]$, 
$\varphi \in \mathscr{B}_{\gamma}$ and $l \in \bb R$,
\begin{align}\label{Esti_I32_Bound}
\left|  I_{32} - \sqrt{ 2 \pi } \pi_s(\varphi)  \right| 
\leq  \left( \frac{ct}{\sqrt{n}}  +  \frac{c_{\ee}}{\sqrt{n}} +  \frac{c}{t^2}  \right)  \| \varphi \|_{\infty}
     + c \left( |l| \sqrt{n} + \frac{1}{\sqrt{n}} \right) \|\varphi\|_{\gamma}.  
\end{align}
Combining \eqref{EstimI1n00a}, \eqref{Esti I2n},  \eqref{EstiJ11} and \eqref{Esti_I32_Bound}, 
we conclude the proof of Proposition \ref{KeyPropo}.
\end{proof}

We proceed to give an asymptotic expansion of the perturbed operator $R_{s, it}$ when $s$ is negative. 
For $s <0$, let 
\begin{align*}
\phi_s(w) = e^{-sw} \mathds{1}_{\{ w \leq 0 \} },  \quad  w \in \bb{R}. 
\end{align*}
With the notation in \eqref{smoo001}, for $\ee \in (0, 1)$, 
the function $\phi_{s,\ee}^{+}$ is given as follows: 
$\phi_{s,\ee}^{+}(w) = 0$ when $w > \ee$; 
$\phi_{s,\ee}^{+}(w) = 1$ when $w \in [-\ee, \ee]$;
$\phi_{s,\ee}^{+}(w) = e^{-s(w + \ee)}$ when $w < -\ee$.  
So the Fourier transform of $\phi_{s,\ee}^{+}$ is 
\begin{align}\label{Diffephi}
 \widehat {\phi}^+_{s,\ee}(u)  
=  \int_{\bb{R}} e^{-iuw} \phi_{s,\ee}^{+}(w) dw
= 2 \frac{\sin (\ee u)}{u}   
    + e^{i \ee u }  \frac{1}{-s - i u },  \quad   u \in \bb{R},   
\end{align}
where we use the convention that $\frac{\sin (\ee \cdot 0)}{0} = \ee$. 
In the sequel, for any fixed $t >1$, we choose $s<0$  satisfying the equation:
\begin{align}\label{SaddleEqua-bis}
\Lambda'(s) - \Lambda'(0) = - \frac{\sigma t}{\sqrt{n}}. 
\end{align}

The following result is an analogue of Proposition \ref{KeyPropo}
and will be used to establish the lower tail Cram\'er type moderate deviation expansion
 \eqref{LD general lower001}. 
As in Proposition \ref{KeyPropo}, the explicit dependence on $\ee,  t,  l$ and $\varphi$ will play a crucial role.  

\begin{proposition}\label{KeyPropo-02}
Assume \ref{Ch7Condi-Moment} and \ref{Ch7Condi-IP}. 
Let $\widehat {\phi}^+_{s,\ee}$ be defined in \eqref{Diffephi}. 
Suppose that $s<0$ satisfies the equation \eqref{SaddleEqua-bis}. 
Let $(t_n)_{n \geq 1}$ be a sequence of positive numbers such that $t_n \to \infty$ and $\frac{t_n}{\sqrt{n}} \to 0$ as $n \to \infty$. 
Then, there exist constants $\gamma, s_0, c, c_{\ee} > 0$ such that 
for any $\ee \in (0,1)$, $s \in (-s_0, 0)$, $x\in \bb P(V)$, $t \in [t_n, o(\sqrt{n} )]$, 
$\varphi \in \mathscr{B}_{\gamma}$ and $l \in \bb R$, 
\begin{align*} 
&  \left|  -s \sigma_s \sqrt{n}  \int_{\bb{R}} e^{-iuln} R_{s, iu}^n (\varphi)(x) 
\widehat{\phi}^+_{s,\ee}(u) \widehat\rho_{\ee^{2}}(u) du 
 -  \sqrt{ 2 \pi }  \pi_s(\varphi)  \right|  \nonumber\\
&  \leq  \left( \frac{ct}{\sqrt{n}}  +  \frac{c_{\ee}}{\sqrt{n}} +  \frac{c}{t^2}  \right)  \| \varphi \|_{\infty}
     + c \left( |l| \sqrt{n} + \frac{1}{\sqrt{n}}  +  e^{-c_{\ee} n } \right) \|\varphi\|_{\gamma}.  
\end{align*}
\end{proposition}

\begin{proof}
Since the proof of Proposition \ref{KeyPropo-02} can be carried out 
in an analogous way as that of Proposition \ref{KeyPropo}, we only sketch the main differences.

Without loss of generality, we assume that the target function $\varphi$ is non-negative. 
From Lemma \ref{Ch7perturbation thm}, we have the following decomposition:
with $\delta > 0$ small enough,
\begin{align} \label{Pro_Neg_s_I}
-s \sigma_s \sqrt{n} \int_{\mathbb{R}} e^{-iuln}  R_{s, iu}^n (\varphi)(x) 
\widehat{\phi}^+_{s,\ee}(u) \widehat\rho_{\varepsilon^{2}}(u) du  
 = I_1 + I_2 + I_3, 
\end{align}
where 
\begin{align*} 
I_1  & =  -s \sigma_s \sqrt{n} \int_{|u|\geq\delta}  e^{-iuln}  R^{n}_{s, iu}(\varphi)(x) 
              \widehat{\phi}^+_{s,\ee}(u)  \widehat\rho_{\varepsilon^{2}}(u) du,  \nonumber\\
I_2  & =  -s \sigma_s \sqrt{n}  \int_{|u|< \delta }  e^{-iuln}  N^{n}_{s, iu}(\varphi )(x) 
               \widehat{\phi}^+_{s,\ee}(u)  \widehat\rho_{\varepsilon^{2}}(u) du,   \nonumber\\
I_3  & =  -s \sigma_s \sqrt{n}  \int_{|u|<\delta}  e^{-iuln}  \lambda^{n}_{s, iu} \Pi_{s, iu}(\varphi )(x) 
              \widehat{\phi}^+_{s,\ee}(u)  \widehat\rho_{\varepsilon^{2}}(u) du.        
\end{align*}
Similarly to the proof of \eqref{equation t and s 1}, from \eqref{SaddleEqua-bis} one can verify that 
\begin{align}\label{equation_s_02}
- \frac{ \sigma t }{ \sqrt{n} } = \sum_{k=2}^{\infty} \frac{ \gamma_k }{ (k-1)! } s^{k-1},
\end{align}
where $\gamma_k = \Lambda^{(k)}(0).$
For any $t>1$ and sufficiently large $n$, 
the equation \eqref{equation_s_02} has a unique solution given by
\begin{align}\label{root_s_2}
s = \frac{1}{\gamma_2^{ 1/2 }} \left( -\frac{t}{\sqrt{n}} \right) 
    - \frac{\gamma_3}{ 2 \gamma_2^2 } \left( -\frac{t}{ \sqrt{n} }  \right)^2
    - \frac{\gamma_4\gamma_2-3\gamma_3^2}{6\gamma_2^{7/2} } \left( -\frac{t}{\sqrt{n}} \right)^3
    + \cdots. 
\end{align}
The series on the right-hand side of \eqref{root_s_2} is absolutely convergent,
and we can assume that $s \in (-s_0, 0)$ for sufficiently small constant $s_0>0$.

\textit{Estimate of $I_1$.} 
From \eqref{Diffephi} and the fact that $\rho_{\ee^2}$ is a density function on $\bb R$, we see that 
\begin{align}\label{Contrpsirho_bb}
\sup_{u \in \bb{R}} |\widehat {\phi}^-_{s,\ee}(u)|
\leq \widehat {\phi}^-_{s,\ee}(0) = \frac{1}{-s} + 2 \ee, 
\quad  
\sup_{u \in \bb{R}} |\widehat\rho_{\ee^{2}}(u)| 
\leq \widehat\rho_{\ee^{2}}(0) = 1.  
\end{align}
From \eqref{EstiRsit} and  \eqref{Contrpsirho_bb}, 
the desired bound for $I_1$ follows: 
\begin{align}\label{EstimI1n00_b}
|I_1| \leq  c e^{-c_{\ee} n } \|\varphi \|_{\gamma}. 
\end{align}

\textit{Estimate of $I_2$.}
Using the bound \eqref{Ch7SpGapContrN} and \eqref{Contrpsirho_bb}, one has
\begin{align}\label{Esti_I2n_bb}
 |I_2|  \leq  C e^{-cn}  \|\varphi \|_{\gamma}. 
\end{align}

\textit{Estimate of $I_3$.}
For brevity, we denote for any $s \in (-s_0, 0)$ and $x \in \bb P(V)$,
\begin{align}\label{Defpsisxt_bb}
\Psi_{s,x}(u):= \Pi_{s, iu}(\varphi)(x) 
\widehat {\phi}^+_{s,\varepsilon}(u) \widehat\rho_{\varepsilon^{2}}(u),  \quad  -\delta < u < \delta. 
\end{align}
Recalling that $K_s(iu)= \log \lambda_{s,iu}$, we decompose the term $I_3$ into two parts: 
\begin{align}  \label{DecompoJ1_bis}
I_3 
=  - s \sigma_s \sqrt{n}  \int_{ |u| < \delta  }  e^{n K_s( iu) - iu ln} \Psi_{s,x}(u) du
=  I_{31} + I_{32}, 
\end{align}
where 
\begin{align*}
&   I_{31} = -s \sigma_s \sqrt{n}  
  \int_{ n^{- \frac{1}{2}} \log n \leq |u| < \delta  }  e^{n K_s( iu) - iu ln} \Psi_{s,x}(u) du,  \nonumber\\
&   I_{32}  = -s \sigma_s \sqrt{n}  \int_{ |u| <  n^{-\frac{1}{2}} \log n }  e^{n K_s( iu) -  iu ln} \Psi_{s,x}(u) du. 
\end{align*}

\textit{Estimate of $I_{31}$.}
By \eqref{Ch7SpGapContrN} and \eqref{Contrpsirho_bb},  
there exists a constant $c >0$ such that for all $s \in (-s_0, 0)$, $|u| < \delta$, $x \in \bb P(V)$
and $\varphi \in \mathscr{B}_{\gamma}$,  
\begin{align} \label{Estipsi01_bis}
|\Psi_{s,x}(u)|  \leq  \frac{c}{-s} \|\varphi\|_{\gamma}.
\end{align}
Similarly to the proof of \eqref{EstiJ11}, it follows that 
there exists a constant $c >0$ such that for all $s \in (-s_0, 0)$, $x \in \bb P(V)$
and $\varphi \in \mathscr{B}_{\gamma}$,  
\begin{align} \label{EstiJ11_bis}
|I_{31}| \leq  c \sqrt{n}  \|\varphi\|_{\gamma}
\int_{ n^{ - \frac{1}{2} }  \log n  \leq |u| < \delta  } e^{- c n u^2 }  du  
\leq \frac{c}{ \sqrt{n} }  \|\varphi\|_{\gamma}. 
\end{align}

\textit{Estimate of $I_{32}$.}
Similarly to \eqref{I32-decomposition-yy}, by a change of variable $u' = u \sigma_s \sqrt{n} $, we get
\begin{align*}
 I_{32}  &  =  -s \int_{- \sigma_s \log n}^{\sigma_s \log n }
  e^{-\frac{u^2}{2}}  \exp \left\{ \sum_{k=3}^{\infty} \frac{ \Lambda^{(k)}(s) (iu)^k}{ \sigma_s^k \, k! \, n^{k/2-1}} \right\}  
  e^{-i \frac{l \sqrt{n}}{\sigma_s}  u } \Psi_{s,x} \left( \frac{u}{\sigma_s \sqrt{n}} \right) du  \nonumber\\
& = (-s) \int_{- \sigma_s \log n}^{\sigma_s \log n }  e^{-\frac{u^2}{2}}  
  \left[ \exp \left\{ \sum_{k=3}^{\infty} \frac{ \Lambda^{(k)}(s) (iu)^k}{ \sigma_s^k \, k! \, n^{k/2-1}} \right\} - 1 \right] 
  e^{-i \frac{l \sqrt{n}}{\sigma_s}  u }  \Psi_{s,x} \left( \frac{u}{\sigma_s \sqrt{n}} \right) du  \nonumber\\ 
& \qquad  +  (-s) \int_{- \sigma_s \log n}^{\sigma_s \log n }  e^{-\frac{u^2}{2}}   
  \left[ e^{-i \frac{l \sqrt{n}}{\sigma_s}  u }  - 1 \right] 
   \Psi_{s,x} \left( \frac{u}{\sigma_s \sqrt{n}} \right) du  \nonumber\\ 
& \qquad  +  (-s) \int_{- \sigma_s \log n}^{\sigma_s \log n }  e^{-\frac{u^2}{2}}    
   \Psi_{s,x} \left( \frac{u}{\sigma_s \sqrt{n}} \right) du  \nonumber\\
& = : I_{321} + I_{322} + I_{323}. 
\end{align*}
For $I_{321}$ and $I_{322}$, 
in the same way as in the proof of \eqref{Esti_I321_aa} and \eqref{Esti_I322_aa}, we have
\begin{align}\label{Esti_I321_I322_bb}
|I_{321}| \leq \frac{c}{\sqrt{n}} \|\varphi\|_{\gamma},
\qquad 
|I_{322}|  \leq  c |l| \sqrt{n} \|\varphi\|_{\gamma}. 
\end{align}
For $I_{323}$, we shall establish the following bound: there exist constants $c, c_{\ee} > 0$ such that
for all $s \in (-s_0, 0)$, $x\in \bb P(V)$, $t \in [t_n, o(\sqrt{n} )]$ and  
$\varphi \in \mathscr{B}_{\gamma}$, 
\begin{align}\label{Esti_I32_bb_02}
\left|  I_{323} -   \sqrt{ 2 \pi } \pi_s(\varphi)  \right| 
\leq   \left( \frac{ct}{\sqrt{n}}  +  \frac{c_{\ee}}{\sqrt{n}} +  \frac{c}{t^2}  \right)  \| \varphi \|_{\infty}
   +  \frac{c}{\sqrt{n}} \| \varphi \|_{\gamma}. 
\end{align}
For brevity, denote $u_n = \frac{u}{ \sigma_s \sqrt{n} }.$
In view of \eqref{Defpsisxt_bb}, we write 
\begin{align*}
\Psi_{s,x} (u_n ) =  h_1(u_n) + h_2(u_n) + h_3(u_n) + h_4(u_n), 
\end{align*}
where
\begin{align*}
&   h_1(u_n) =  \big[ \Pi_{s, i u_n}(\varphi)(x) -  \pi_s(\varphi) \big] 
\widehat {\phi}^+_{s,\varepsilon}(u_n)
\widehat\rho_{\varepsilon^{2}}(u_n),    \nonumber\\
&   h_2(u_n)  =    \pi_s(\varphi) \widehat {\phi}^+_{s,\varepsilon}(u_n)
\left[ \widehat\rho_{\varepsilon^{2}}(u_n) -  \widehat\rho_{\varepsilon^{2}}(0)  \right]  \nonumber\\
&   h_3(u_n) =   \pi_s(\varphi) 
 \left[ \widehat {\phi}^+_{s,\varepsilon}(u_n) - \widehat {\phi}^+_{s,\varepsilon}(0) \right]
   \widehat\rho_{\varepsilon^{2}}(0),  \nonumber\\
&   h_4(u_n) =   \pi_s(\varphi) \widehat {\phi}^+_{s,\varepsilon}(0) \widehat\rho_{\varepsilon^{2}}(0). 
\end{align*}
Then $I_{323}$ can be decomposed into four parts: 
\begin{align}\label{DecompoJ122_b}
I_{323} = J_{1} + J_{2} + J_{3} + J_{4}, 
\end{align}
where for $j = 1,2,3,4$, 
\begin{align*}
J_{j} =  -s  \int_{- \sigma_s \log n }^{ \sigma_s \log n } e^{-\frac{u^2}{2}} 
h_j(u_n) du. 
\end{align*}

\textit{Estimates of $J_1$ and $J_2$.}
Similarly to the proof of \eqref{Controlh1} and \eqref{Controlh2},
one can verify that 
\begin{align}\label{Esti_J1_J2_bb}
| J_1 | \leq  \frac{c}{ \sqrt{n} } \| \varphi \|_{\gamma},
\qquad 
|J_2| \leq \frac{c}{\varepsilon^4}  \frac{1}{ \sqrt{n} } \|\varphi\|_{\infty}
 \leq  \frac{c_{\ee}}{ \sqrt{n} } \|\varphi\|_{\infty}.
\end{align}

\textit{Estimate of $J_{3}$.}
By the definition of the function $\widehat {\phi}^+_{s,\varepsilon}$ (see \eqref{Diffephi}), we have 
\begin{align*}
 \widehat {\phi}^+_{s,\varepsilon}(u_n) - \widehat {\phi}^+_{s,\varepsilon}(0) 
&  = 2 \left( \frac{\sin \ee u_n}{u_n} - \ee \right)
    + \left( e^{i \ee u_n} \frac{1}{-s - i u_n} - \frac{1}{-s} \right) \nonumber\\
& =  2 \left( \frac{\sin \ee u_n}{u_n} - \ee \right) +
      \frac{ e^{i \ee u_n} - 1 }{-s - i u_n}
     + \left( \frac{1}{-s - i u_n} - \frac{1}{-s}  \right)   \nonumber\\
& =: A_1 (u)  +  A_2 (u) + A_3 (u). 
\end{align*}
Since $\widehat\rho_{\varepsilon^{2}}(0) = 1$, it follows that $J_3 = J_{31} + J_{32} + J_{33}$, 
where  
\begin{align*}
J_{3j} = - s \pi_s(\varphi) 
\int_{- \sigma_s \log n }^{ \sigma_s \log n } e^{-\frac{u^2}{2}}  A_j(u) du,  \quad  j = 1, 2, 3.  
\end{align*}
For $J_{31}$, since $| A_1 (u) | \leq c |u|^2/n$ and $|\pi_s(\varphi)| \leq c\|\varphi\|_{\infty}$, 
we have
\begin{align}\label{Esti_J31_bb}
| J_{31} | \leq c \frac{-s}{n} \|\varphi\|_{\infty} \leq \frac{C}{n} \|\varphi\|_{\infty}. 
\end{align} 
For $J_{32}$,
using the inequality $|e^z - 1| \leq e^{\Re z} |z|$, $z \in \bb C$, we get 
\begin{align*}
|A_2 (u)|  \leq  \frac{ \ee |u_n| }{|-s - i u_n|} \leq  \frac{ \ee |u_n| }{-s}. 
\end{align*}
Hence, 
\begin{align}\label{Esti_J32_bb}
|J_{32}| \leq \frac{c}{ \sqrt{n} } \|\varphi\|_{\infty}. 
\end{align}
For $J_{33}$, note that $A_3 (u) = \frac{-is u_n - u_n^2}{-s(s^2 + u_n^2)}$ and $s = O(\frac{-t}{\sqrt{n}})$. 
Using the fact that the integral of an odd function over a symmetric interval is identically zero, 
by elementary calculations we derive that 
\begin{align}\label{Esti_J33_bb}
|J_{33}| = \left| \pi_s(\varphi) \int_{- \sigma_s \log n }^{ \sigma_s \log n } e^{-\frac{u^2}{2}}  
     \frac{u_n^2}{s^2 + u_n^2} du \right| 
 \leq \frac{c}{t^2} \|\varphi\|_{\infty}. 
\end{align}
Putting together \eqref{Esti_J31_bb}, \eqref{Esti_J32_bb} and \eqref{Esti_J33_bb}, we obtain
\begin{align} \label{Esti_J3_bb}
J_{3} \leq  \frac{c}{ \sqrt{n} } \|\varphi\|_{\infty} +   \frac{c}{ t^2 }  \|\varphi\|_{ \infty }. 
\end{align}

\textit{Estimate of $J_{4}$.}
Similarly to the proof of \eqref{EstiIntegral00a},
 one has
\begin{align}\label{Esti_J4_bb}
\left|  J_4 - \sqrt{ 2 \pi } \pi_s(\varphi) \right| \leq  \frac{ct}{\sqrt{n}} \| \varphi \|_{\infty}. 
\end{align}
In view of \eqref{DecompoJ122_b}, combining \eqref{Esti_J1_J2_bb}, \eqref{Esti_J3_bb}, 
and \eqref{Esti_J4_bb}, 
we obtain the desired bound \eqref{Esti_I32_bb_02}. 
Putting together the bounds \eqref{EstimI1n00_b}, \eqref{Esti_I2n_bb}, \eqref{EstiJ11_bis}, \eqref{Esti_I321_I322_bb}
and \eqref{Esti_I32_bb_02}, 
we finish the proof of Proposition \ref{KeyPropo-02}.  
\end{proof}

\subsection{Proof of Theorem \ref{Thm-Cram-Entry_bb}} \label{Ch7_Sec_MDE_Entry}

To establish Theorem \ref{Thm-Cram-Entry_bb}, 
we first prove the following moderate deviation expansion in the normal range $y \in [0, o(n^{1/6})]$
for the couple $(G_n \!\cdot\! x, \log |\langle f, G_n v \rangle|)$ with a target function $\varphi$ on $G_n \!\cdot\! x$. 

\begin{theorem}\label{Thm-CramSca-02}
Assume \ref{Ch7Condi-Moment} and \ref{Ch7Condi-IP}. 
Then, there exists a constant $\gamma>0$ such that, as $n \to \infty$,  
uniformly in 
$t \in [0, o(n^{1/6})]$, $\varphi \in \mathscr{B}_{\gamma}$,
$x = \bb R v \in \bb P(V)$ and $y = \bb R f \in \bb P(V^*)$  with $\|v\| = \|f\| =1$,  
\begin{align*}
\qquad \quad  \frac{\bb{E}
\big[ \varphi(G_n \!\cdot\! x) \mathds{1}_{ \{ \log| \langle f,  G_n v \rangle | - n\lambda_1 \geq \sqrt{n} \sigma t \} } \big] }
{ 1-\Phi(t) }    
&  =  \nu(\varphi) +  \| \varphi \|_{\gamma}  o(1),  \nonumber\\
\frac{\bb{E}
\big[ \varphi(G_n \!\cdot\! x) \mathds{1}_{ \{ \log| \langle f,  G_n v \rangle | - n\lambda_1 \leq - \sqrt{n} \sigma t  \} } \big] }
{ \Phi(-t)  }   
&  = \nu(\varphi) +  \| \varphi \|_{\gamma}  o(1).  
\end{align*}
\end{theorem}

To prove Theorem \ref{Thm-CramSca-02}, 
we need the exponential H\"{o}lder regularity of the invariant measure $\nu$ (Lemma \ref{Lem_Regu_pi_s} with $s =0$)
and the following moderate deviation expansion for the norm cocycle 
$\sigma (G_n, x)$ established recently in \cite{XGL19b}.

\begin{lemma}[\cite{XGL19b}] \label{Lem_Cramer_Cocy}
Assume \ref{Ch7Condi-Moment} and \ref{Ch7Condi-IP}. 
Then, there exists a constant $\gamma>0$ such that, as $n \to \infty$,  
uniformly in $x\in \bb P(V)$, $t \in [0, o(\sqrt{n} )]$ and $\varphi \in \mathscr{B}_{\gamma}$, 
\begin{align*} 
\frac{\bb{E} \left[ \varphi(G_n \!\cdot\! x) \mathds{1}_{ \left\{ \sigma(G_n, x) - n \lambda_1 \geq \sqrt{n}\sigma t \right\} } \right] }
{ 1-\Phi(t) }
& =  e^{ \frac{t^3}{\sqrt{n}} \zeta( \frac{t}{\sqrt{n}} ) }
\left[ \nu(\varphi) +  \| \varphi \|_{\gamma} O\left( \frac{t+1}{\sqrt{n}} \right) \right],   \nonumber\\
\frac{\bb{E} \left[ \varphi(G_n \!\cdot\! x) \mathds{1}_{ \left\{ \sigma(G_n, x) - n \lambda_1 \leq - \sqrt{n}\sigma t  \right\} } \right] }
{ 1-\Phi(t) }
& =  e^{ -\frac{t^3}{\sqrt{n}} \zeta( -\frac{t}{\sqrt{n}} ) }
\left[ \nu(\varphi) +  \| \varphi \|_{\gamma} O\left( \frac{t+1}{\sqrt{n}} \right) \right].  
\end{align*}
\end{lemma}

\begin{proof}[Proof of Theorem \ref{Thm-CramSca-02}]
Without loss of generality, we assume that the target function $\varphi$ is non-negative. 
We only show the first expansion in Theorem \ref{Thm-CramSca-02}
since the proof of the second one can be carried out in a similar way.

The upper bound is a direct consequence of Lemma \ref{Lem_Cramer_Cocy}.
Specifically, since $\log |\langle f, G_n v \rangle| \leq \sigma(G_n, x)$
for $x = \bb R v \in \bb P(V)$ and $f \in V^*$ with $\|v\| = \|f\| = 1$, 
using the first expansion in Lemma \ref{Lem_Cramer_Cocy}, 
we get that there exists a constant $c>0$ such that 
for any $t \in [0, o(\sqrt{n} )]$, $\varphi \in \mathscr{B}_{\gamma}$, 
 $x = \bb R v \in \bb P(V)$ and $f \in V^*$ with $\|v\| = \|f\| =1$, 
\begin{align}\label{Pf-Cram-Scal-Upp}
  \frac{\bb{E} \left[ \varphi(G_n \!\cdot\! x) 
  \mathds{1}_{ \left\{ \log| \langle f,  G_n v \rangle | - n\lambda_1 \geq \sqrt{n}\sigma t \right\} } \right] }
{ 1-\Phi(t)  }    
\leq   e^{ \frac{t^3}{\sqrt{n}}\zeta (\frac{t}{\sqrt{n}} ) }
\left[ \nu(\varphi) + c \| \varphi \|_{\gamma} \frac{t +1}{\sqrt{n}} \right]. 
\end{align}

For the lower bound, we use Lemmas \ref{Lem_Regu_pi_s} and \ref{Lem_Cramer_Cocy}. 
Let $\varepsilon>0$. 
By Lemma \ref{Lem_Regu_pi_s} with $s =0$, 
there exist constants $c, C >0$ (depending on $\varepsilon$) such that for all $n \geq k \geq 1$, 
$x = \bb R v \in \bb P(V)$ and $f \in V^*$ with $\|v\| = \|f\| = 1$,
\begin{align*}
\mathbb{P} \Big( \log |\langle f, G_n v \rangle| -  \sigma(G_n, x) \leq -\ee k \Big) \leq C e^{ -c k }. 
\end{align*}
Using this inequality, we get
\begin{align}\label{Pf_Cra_ine_jjj}
& \bb{E} \left[ \varphi(G_n \!\cdot\! x) 
  \mathds{1}_{ \left\{ \log| \langle f,  G_n v \rangle | - n\lambda_1 \geq \sqrt{n}\sigma t \right\} } \right]  \nonumber\\
 & \geq  \mathbb{E} \left[ \varphi(G_n \!\cdot\! x) 
 \mathds{1}_{ \left\{ \log| \langle f,  G_n v \rangle | - n\lambda_1 \geq \sqrt{n}\sigma t \right\} }
 \mathds{1}_{ \left\{ \log |\langle f, G_n v \rangle| -  \sigma(G_n, x)  > -\ee k  \right\} }  \right]  \nonumber\\
 & \geq   \mathbb{E} \left[ \varphi(G_n \!\cdot\! x) 
 \mathds{1}_{ \left\{ \sigma (G_n, x) - n\lambda_1 \geq \sqrt{n}\sigma t + \varepsilon k  \right\} }
 \mathds{1}_{  \left\{ \log |\langle f, G_n v \rangle| -  \sigma(G_n, x) > -\ee k  \right\} }  \right]  \nonumber\\
 &  \geq   \mathbb{E} \left[ \varphi(G_n x) 
 \mathds{1}_{ \left\{ \sigma(G_n, x) - n\lambda_1 \geq \sqrt{n}\sigma t + \ee k  \right\} }  \right]
    - C e^{ -c k } \| \varphi \|_{\infty}.
\end{align}
By Lemma \ref{Lem_Cramer_Cocy}, we have, as $n \to \infty$, 
uniformly in $x \in \bb P(V)$, $t \in [0, \sqrt{\log n}]$ and $\varphi \in \mathscr{B}_{\gamma}$, 
\begin{align}\label{NormUpp02}
\frac{ \mathbb{E} \left[ \varphi(G_n \!\cdot\! x) 
 \mathds{1}_{ \left\{ \sigma(G_n, x) - n\lambda_1 \geq \sqrt{n}\sigma t + \varepsilon k  \right\} }  \right] }
{ 1-\Phi(t_1)  }
=   \nu(\varphi) + \| \varphi \|_{\gamma} O \left( \frac{t_1^3  + 1}{\sqrt{n}} \right), 
\end{align}
where $t_1 = t + \frac{\ee k}{ \sigma \sqrt{n} }$. 
Take $k = \floor{A \log n}$ in \eqref{Pf_Cra_ine_jjj}, 
where $A >0$ is a  constant whose value will be chosen sufficiently large. 
We claim that uniformly in $t \in [0, \sqrt{\log n}]$, 
\begin{align}\label{Inequ_Normal_aaa}
1 > \frac{1 - \Phi(t_1) }{1 - \Phi(t) }
  = 1 - \frac{ \int_{t}^{t_1} e^{-\frac{u^2}{2}} du }{ \int_{t}^{\infty} e^{-\frac{u^2}{2}} du } 
  > 1 - c \frac{ t + 1 }{\sqrt{n}} \log n.  
\end{align}
Indeed, when $y \in [0,2]$, the inequality \eqref{Inequ_Normal_aaa} holds
due to the fact that $t_1 = t + \frac{\ee k}{ \sigma \sqrt{n} }$ and $k = \floor{A \log n}$;
when $t \in [2, \sqrt{\log n}]$, we can use the inequality
$e^{\frac{t^2}{2}} \int_{t}^{\infty} e^{-\frac{u^2}{2}} du \geq \frac{1}{t} - \frac{1}{t^3} > \frac{1}{2t}$ to get
\begin{align*}
1 - \frac{ \int_{t}^{t_1} e^{-\frac{u^2}{2}} du }{ \int_{t}^{\infty} e^{-\frac{u^2}{2}} du } 
  > 1 - \frac{ (t_1 - t) e^{- t^2/2} }{\frac{1}{2t} e^{- t^2/2}}
  > 1 - c \frac{ t + 1 }{\sqrt{n}} \log n, 
\end{align*}
so that \eqref{Inequ_Normal_aaa} also holds. 
It is easy to check that $\frac{t_1^3  + 1}{\sqrt{n}} = O( \frac{t^3 + 1}{\sqrt{n}} )$,
uniformly in $t \in [0, \sqrt{\log n}]$. 
Consequently, there exists a constant $c >0$ such that for any 
$x \in \bb P(V)$, $t \in [0, \sqrt{\log n}]$ and $\varphi \in \mathscr{B}_{\gamma}$, 
\begin{align*}
\frac{ \mathbb{E} \left[ \varphi(G_n \!\cdot\! x) 
 \mathds{1}_{ \left\{ \sigma(G_n, x) - n\lambda_1 \geq \sqrt{n}\sigma t + \varepsilon k  \right\} }  \right] }
{ 1-\Phi(t)  }
\geq  \nu(\varphi) -  c \| \varphi \|_{\gamma} \frac{ t + 1 }{\sqrt{n}} \log n.
\end{align*}
This, together with \eqref{Pf_Cra_ine_jjj} and the fact that $e^{-ck}/[1 - \Phi(t)]$ decays to $0$ faster than $\frac{1}{n}$
(by taking $A>0$ to be sufficiently large), 
implies that there exists a constant $c >0$ such that for any $t \in [0, \sqrt{\log n}]$, $\varphi \in \mathscr{B}_{\gamma}$, 
$x = \bb R v \in \bb P(V)$ and $f \in V^*$ with $\|v\| = \|f\| =  1$,
\begin{align} \label{Pf_Cram_Low_logn}
\frac{\bb{E} \left[ \varphi(G_n \!\cdot\! x) 
  \mathds{1}_{ \left\{ \log| \langle f,  G_n v \rangle | - n\lambda_1 \geq \sqrt{n}\sigma t  \right\} } \right] }
{ 1-\Phi(t)  }   
\geq   \nu(\varphi) - c \| \varphi \|_{\gamma} \frac{t + 1}{\sqrt{n}} \log n. 
\end{align}

It remains to prove the lower bound when $t \in [\sqrt{\log n}, o(n^{1/6})]$. 
In the same way as in \eqref{NormUpp02}, we get that, with $t_1 = t + \frac{\ee k}{ \sigma \sqrt{n} }$, 
uniformly in $x \in \bb P(V)$, $t \in [\sqrt{\log n}, o(n^{1/6})]$ and $\varphi \in \mathscr{B}_{\gamma}$, 
\begin{align}\label{NormUpp_yy}
\frac{ \mathbb{E} \left[ \varphi(G_n \!\cdot\! x) 
 \mathds{1}_{ \left\{ \sigma(G_n, x) - n\lambda_1 \geq \sqrt{n}\sigma t + \varepsilon k  \right\} }  \right] }
{ 1-\Phi(t_1)  }
=   \nu(\varphi) + \| \varphi \|_{\gamma} O \left( \frac{t_1^3  + 1}{\sqrt{n}} \right).
\end{align}
We take $k = \floor{A t^2}$ in \eqref{Pf_Cra_ine_jjj}, where $A >0$ is a fixed sufficiently large constant. 
Using the inequality 
$\frac{1}{t} \geq e^{\frac{t^2}{2}} \int_{t}^{\infty} e^{-\frac{u^2}{2}} du \geq \frac{1}{t} - \frac{1}{t^3}$
for $t > 0$, by elementary calculations, 
we get that uniformly in $t \in [\sqrt{\log n}, o(n^{1/6})]$, 
\begin{align*}
1 > \frac{1 - \Phi(t_1) }{1 - \Phi(t) }
 =   \frac{ \int_{t_1}^{\infty} e^{-\frac{u^2}{2}} du }{ \int_{t}^{\infty} e^{-\frac{u^2}{2}} du }
& \geq  \frac{ ( \frac{1}{t_1} - \frac{1}{t_1^3} ) e^{ - \frac{t_1^2}{2} }  }{ \frac{1}{t} e^{ - \frac{t^2}{2} } }
  =   \frac{t}{t_1} \left( 1 - \frac{1}{t_1^2} \right) e^{ \frac{t^2}{2} - \frac{t_1^2}{2} }   \notag\\
&  >  \frac{1}{ 1 + \frac{\ee k}{ \sigma \sqrt{n} t } }  \left( 1 - \frac{1}{t^2} \right) 
       e^{ - \frac{\ee k}{ \sigma \sqrt{n} }t - \frac{\ee^2 k^2}{ 2 \sigma^2 n } }. 
\end{align*}
Using the inequalities $\frac{1}{1+x} \geq 1 - x$ and $e^{-x} \geq 1 -x$ for $x \geq 0$,
and taking into account that $k = \floor{A t^2}$, we get that, as $n \to \infty$, 
uniformly in $t \in [\sqrt{\log n}, o(n^{1/6})]$, 
\begin{align}\label{Rate-Phi-t}
1 > \frac{1 - \Phi(t_1) }{1 - \Phi(t) } \geq 
 \left( 1 -  \frac{\ee k}{ \sigma \sqrt{n} t }  \right) \left( 1 - \frac{1}{t^2} \right)  
    \left( 1 - \frac{\ee k}{ \sigma \sqrt{n} }t - \frac{\ee^2 k^2}{ 2 \sigma^2 n } \right)
    = 1 - |o(1)|.  
\end{align}
Taking into account that $\frac{t_1^3  + 1}{\sqrt{n}} = O( \frac{t^3  + 1}{\sqrt{n}} )$
and that $e^{-c A t^2}/[1 - \Phi(t)]$ decays to $0$ faster than $\frac{1}{n}$ (by taking $A>0$ to be sufficiently large),
from \eqref{Pf_Cra_ine_jjj}, \eqref{NormUpp_yy} and \eqref{Rate-Phi-t}
we deduce that there exists a constant $c >0$ such that for any
 $t \in [\sqrt{\log n}, o(n^{1/6})]$, $\varphi \in \mathscr{B}_{\gamma}$, 
$x = \bb R v \in \bb P(V)$ and $f \in V^*$ with $\|v\| = \|f\| = 1$,  
\begin{align} \label{Pf_Cram_Low_n16}
\frac{\bb{E} \left[ \varphi(G_n \!\cdot\! x) 
  \mathds{1}_{ \left\{ \log| \langle f,  G_n v \rangle | - n\lambda_1 \geq \sqrt{n}\sigma t  \right\} } \right] }
{ 1-\Phi(t)  }   
\geq   \nu(\varphi) -  \| \varphi \|_{\gamma} |o(1)|. 
\end{align}
Combining \eqref{Pf-Cram-Scal-Upp}, \eqref{Pf_Cram_Low_logn} and \eqref{Pf_Cram_Low_n16} 
finishes the proof of Theorem \ref{Thm-CramSca-02}. 
\end{proof}

As in \cite{DKW21, XGL22a}, we shall use a partition of the unity on the projective space $\bb P(V)$.  
Let $U$ be the uniform distribution function on the interval $[0,1]$, namely,
$U(t)=t$ for $t\in [0,1]$, $U(t)=0$ for $t <0$ 
and $U(t)=1$ for $t > 1$. 
 Let $a \in (0,\frac{1}{2}]$ be a constant.
For any integer $k\geq 0$, define 
\begin{align*} 
U_{k}(t)= U\left(\frac{t-(k-1) a}{a}\right),  \qquad 
h_{k}(t)=U_{k}(t) - U_{k+1}(t),  \quad  t \in \bb R. 
\end{align*}
Since $U_{m} = \sum_{k=m}^\infty h_{k}$ for $m\geq 0$,  we have that for any $t\geq 0$ and $m\geq0$,
\begin{align} \label{unity decomposition h-001-222}
\sum_{k=0}^{\infty} h_{k} (t) =1, \quad \sum_{k=0}^{m} h_{k} (t) + U_{m+1} (t) =1.
\end{align}
Set, for any $x \in \bb P(V)$ and $y \in \bb P(V^*)$,  
\begin{align}\label{Def-chi-nk-222}
\chi_{k}^y(x)=h_{k}(-\log \delta(y, x))  \quad  \mbox{and}  \quad 
\overline \chi_{k}^y(x)= U_{k} ( -\log \delta(y, x) ). 
\end{align}
From \eqref{unity decomposition h-001-222}
 we have the following partition of the unity  on $\bb P(V)$: for any $x\in \bb P (V)$, $y \in \bb P(V^*)$ and $m\geq 0$,
\begin{align} \label{Unit-partition001-222}
\sum_{k=0}^{\infty} \chi_{k}^y (x) =1, \quad 
\sum_{k=0}^{m} \chi_{k}^y (x) + \overline \chi_{m+1}^y (x) =1.
\end{align}
Let $\supp (\chi_{k}^y)$ be the support of the function $\chi_{k}^y$.  
It is easy to see that for any $k\geq 0$ and $y\in \bb P(V^*)$,
\begin{align} \label{on the support on chi_k-001-222}
 -\log \delta(y, x) \in [a (k-1), a (k+1)] \quad \mbox{for any}\ x\in \supp (\chi_{k}^y). 
\end{align}
In the same way as in the proof of \cite[Lemma 4.8]{XGL22a}, one can show that
there exists a constant $c>0$ such that 
for any $\gamma\in(0,1]$, 
 $k\geq 0$ and $y\in \bb P(V^*)$, it holds
  $\chi_{k}^y\in \scr B_{\gamma}$ and 
\begin{align} \label{Holder prop ohCHI_k-001-222}
\| \chi_{k}^y \|_{\gamma} \leq \frac{c e^{\gamma k a}}{a^\gamma}.
\end{align}

Now we shall apply Propositions \ref{KeyPropo} and \ref{KeyPropo-02}
to establish the following moderate deviation expansion when $t \in [n^{\alpha}, o(n^{1/2})]$ for any $\alpha \in (0, 1/2)$.

\begin{theorem}\label{Thm-Cram-Scalar-tag}
Assume \ref{Ch7Condi-Moment} and \ref{Ch7Condi-IP}. 
Then, there exists a constant $\gamma>0$ such that
 for any $\varphi \in \mathscr{B}_{\gamma}$ and $\alpha \in (0, 1/2)$,  we have,  as $n \to \infty$, 
uniformly in $t \in [n^{\alpha}, o(\sqrt{n} )]$,
 $x = \bb R v \in \bb P(V)$ and $f \in V^*$ with $\|v\| = \|f\| =1$, 
\begin{align}
\frac{\bb{E} \left[ \varphi(G_n \!\cdot\! x) 
\mathds{1}_{ \left\{ \log| \langle f,  G_n v \rangle | - n\lambda_1 \geq \sqrt{n} \sigma t \right\} } \right] }
{ 1-\Phi(t) }    
& =  e^{ \frac{t^3}{\sqrt{n}}\zeta (\frac{t}{\sqrt{n}} ) }
\Big[ \nu(\varphi) +   o(1) \Big],  \label{Cramer-Coeffi-upper} \\
\frac{\bb{E}  \left[ \varphi(G_n \!\cdot\! x) 
\mathds{1}_{ \left\{ \log| \langle f,  G_n v \rangle | - n\lambda_1 \leq - \sqrt{n} \sigma t \right\} } \right] }
{ \Phi(-t)  }   
& =  e^{ - \frac{t^3}{\sqrt{n}}\zeta (-\frac{t}{\sqrt{n}} ) }
\Big[ \nu(\varphi) +  o(1) \Big].   \label{Cramer-Coeffi-lower}
\end{align}
\end{theorem}

\begin{proof} 
We first establish \eqref{Cramer-Coeffi-upper}.
Since we have shown the upper bound \eqref{Pf-Cram-Scal-Upp},
 it remains to establish the following lower bound: there exists a constant $\gamma>0$ such that
 for any positive function $\varphi \in \mathscr{B}_{\gamma}$ and $\alpha \in (0, 1/2)$,  
uniformly in $t \in [n^{\alpha}, o(\sqrt{n} )]$,
 $x = \bb R v \in \bb P(V)$ and $y = \bb R f \in \bb P(V^*)$ with $\|v\| = \|f\| =1$, 
\begin{align}\label{ScalLowerBound}
\liminf_{n \to \infty}  
\frac{\mathbb{E} \left[ \varphi(G_n \!\cdot\! x) 
  \mathds{1}_{ \left\{ \log| \langle f,  G_n v \rangle | - n \lambda_1 \geq \sqrt{n} \sigma t  \right\} } \right] }
{ e^{ \frac{t^3}{\sqrt{n}} \zeta(\frac{t}{\sqrt{n}}) } [1-\Phi(t)] }   
\geq  \nu(\varphi). 
\end{align}
Now we are going to prove \eqref{ScalLowerBound}. 
From the change of measure formula \eqref{Ch7basic equ1}
and the fact $\lambda_1 = \Lambda'(0)$, we get 
\begin{align} 
A_n: &  =  \bb{E} \left[ \varphi(G_n \!\cdot\! x) 
  \mathds{1}_{ \left\{ \log| \langle f,  G_n v \rangle | - n \lambda_1 \geq \sqrt{n}\sigma t \right\} } \right]
 \label{ChanMeaScal}\\
& =  r_s(x) \kappa^{n}(s) \bb{E}_{\bb{Q}^{x}_{s}}
\left[ (\varphi r_s^{-1})(G_n \!\cdot\! x) 
e^{-s \sigma (G_n, x) }\mathds{1}_{ \left\{ \log |\langle f, G_n v \rangle | \geq n \Lambda'(0) + \sqrt{n} \sigma t  \right\} }  
\right].    \nonumber
\end{align}
For brevity, we denote 
\begin{align*}
T_n^x = \sigma (G_n, x) - n\Lambda'(s),  \qquad   Y_n^{x,y}: = \log \delta(y, G_n \!\cdot\! x). 
\end{align*}
Choosing $s>0$ as the solution of the equation \eqref{SaddleEqua}, from \eqref{ChanMeaScal} and \eqref{Basic-scalar-decom} it follows that 
\begin{align*}
A_n = r_s(x) e^{-n [ s \Lambda'(s) - \Lambda(s) ] }  
   \bb{E}_{\bb{Q}^{x}_{s}}
 \left[ (\varphi r_s^{-1})(G_n \!\cdot\! x) e^{-s T_n^x }
   \mathds{1}_{ \left\{ T_n^x + Y_n^{x,y} \geq 0 \right\} } \right]. 
\end{align*}
Using \eqref{SaddleEqua}, one can verify that 
\begin{align} \label{LambdaQ01}
s \Lambda'(s) - \Lambda(s) 
= \frac{t^2}{2n} - \frac{t^3}{n^{3/2} } \zeta \Big( \frac{t}{\sqrt{n}} \Big), 
\end{align}
where $\zeta$ is the Cram\'{e}r series defined by \eqref{Ch7Def-CramSeri}. 
Thus $A_n$ can be rewritten as  
\begin{align}\label{An-lambdaq}  
A_n = r_s(x) e^{ -\frac{t^2}{2} + \frac{t^3}{ \sqrt{n} } \zeta(\frac{t}{\sqrt{n}}) }
   \bb{E}_{\bb{Q}^{x}_{s}}
\left[ (\varphi r_s^{-1})(G_n \!\cdot\! x)  e^{-s T_n^x }
\mathds{1}_{ \left\{ T_n^x + Y_n^{x,y} \geq 0 \right\} } \right]. 
\end{align}
Since the functions $\varphi$ and $r_s$ are positive, 
using the partition of the unity \eqref{Unit-partition001-222} and \eqref{on the support on chi_k-001-222}, 
we have that for $M_n = \floor{\log n}$ and $a \in (0, 1/2)$, 
\begin{align}\label{Ch7SmoothIne Holder 01}
A_{n} & \geq   r_s(x) e^{ -\frac{t^2}{2} + \frac{t^3}{ \sqrt{n} } \zeta(\frac{t}{\sqrt{n}}) }
  \sum_{k = 0}^{ M_n } \bb{E}_{\bb{Q}^{x}_{s}}
\left[ (\varphi r_s^{-1} \chi_{k}^y)(G_n \!\cdot\! x)  e^{-s T_n^x }
\mathds{1}_{\{ T_n^x + Y_n^{x,y} \geq 0 \}}  \right]  \nonumber\\
& \geq   r_s(x) e^{ -\frac{t^2}{2} + \frac{t^3}{ \sqrt{n} } \zeta(\frac{t}{\sqrt{n}}) }
  \sum_{k = 0}^{M_n} \bb{E}_{\bb{Q}^{x}_{s}}
\left[ (\varphi r_s^{-1} \chi_{k}^y)(G_n \!\cdot\! x) 
e^{-s T_n^x } \mathds{1}_{\{ T_n^x -  a (k+1) \geq 0 \}} \right].  
\end{align}
Let
\begin{align*}
\varphi_{s,k}^y(x) = (\varphi r_s^{-1} \chi_{k}^y )(x), 
\quad  x \in \bb P(V), 
\end{align*}
and
\begin{align*}
\psi_s(u) = e^{-s u} \mathds{1}_{ \{ u \geq 0 \} },  \quad  u \in \bb{R}. 
\end{align*}
It then follows from \eqref{Ch7SmoothIne Holder 01} that 
\begin{align*} 
A_{n} 
& \geq   r_s(x) e^{ -\frac{t^2}{2} + \frac{t^3}{ \sqrt{n} } \zeta(\frac{t}{\sqrt{n}}) }
  \sum_{k = 0}^{M_n} e^{-s a (k+1)} \bb{E}_{\bb{Q}^{x}_{s}}
\left[ \varphi_{s,k}^y(G_n \!\cdot\! x) \psi_s(T_n^x - a (k+1) )  \right],    
\end{align*}
which implies that
\begin{align*}
& \frac{ A_n }{ e^{ \frac{t^3}{\sqrt{n}} \zeta(\frac{t}{\sqrt{n}}) } [1-\Phi(t)] }   
 \geq  r_s(x) \sum_{k = 0}^{ M_n }   
\frac{  e^{-s a (k+1)}  \bb{E}_{\bb{Q}^{x}_{s}}
\big[ \varphi_{s,k}^y(G_n \!\cdot\! x) \psi_s(T_n^x - a (k+1))  \big] }
   { e^{ \frac{t^2}{2} } [1-\Phi(t)] }. 
\end{align*}
From Lemma \ref{estimate u convo}, it follows that for any small constant $\ee >0$, 
\begin{align*}
\psi_s(u) \geq 
\psi^-_{s,\ee} * \rho_{\ee^2}(u) - 
\int_{ |w| \geq \ee } \psi^-_{s,\varepsilon}( u-w ) \rho_{\varepsilon^2}(w) dw, \quad u \in \bb R,
\end{align*}
where $\psi^-_{s,\ee}$ is given by \eqref{Def-psi-see}. 
Using this inequality we get
\begin{align}\label{LowerScalarACD}
\frac{ A_n }{ e^{ \frac{t^3}{\sqrt{n}} \zeta(\frac{t}{\sqrt{n}}) } [1-\Phi(t)] } 
  \geq  r_s(x) \sum_{k = 0}^{ M_n }
   \frac{ B_{n,k} - D_{n,k} }{ e^{ \frac{t^2}{2} } [1-\Phi(t)] },   
\end{align}
where 
\begin{align*}
&  B_{n,k} =  e^{-s a (k+1)} \mathbb{E}_{\mathbb{Q}^{x}_{s}}
  \Big[ \varphi_{s,k}^y (G_n \!\cdot\! x)
  ( \psi^-_{s,\ee} * \rho_{\ee^2}) (T_{n}^x - a (k+1)  ) \Big],    \nonumber \\
&  D_{n,k} =  e^{-s a (k+1)} \int_{|w| \geq \ee} 
    \mathbb{E}_{\mathbb{Q}^{x}_{s}}
  \Big[ \varphi_{s,k}^y (G_n \!\cdot\! x) \psi^-_{s,\ee}(T_n^x - a (k+1) - w) \Big] 
    \rho_{\ee^2}(w) dw. 
\end{align*}
Since $B_{n,k} \geq D_{n,k}$ and $\sup_{ x \in \bb P(V) } |r_s (x) - 1| \to 0$ as $n \to \infty$, 
by Fatou's lemma, we get
\begin{align*}
\liminf_{n \to \infty}
\frac{ A_n }{ e^{ \frac{t^3}{\sqrt{n}} \zeta(\frac{t}{\sqrt{n}}) } [1-\Phi(t)] }
\geq \sum_{k = 0}^{\infty} \liminf_{n \to \infty} 
   \frac{ B_{n,k} - D_{n,k} }{ e^{ \frac{t^2}{2} } [1-\Phi(t)] }  \mathds{1}_{ \left\{ k \leq M_n \right\}}. 
\end{align*}

\textit{Estimate of $B_{n,k}$.}
Since the function $\widehat\rho_{\ee^{2}}$ is integrable on $\bb{R}$, 
by the Fourier inversion formula, we have 
\begin{align*}
{\psi}^-_{s,\ee}\!\ast\!\rho_{\ee^{2}}(w)
= \frac{1}{2\pi} \int_{\bb{R}}e^{iuw} \widehat {\psi}^-_{s,\ee}(u) \widehat\rho_{\ee^{2}} (u) du, 
\quad  w \in \bb R. 
\end{align*}
Substituting $w = T_{n}^x - a (k+1)$, 
taking expectation with respect to $\bb{E}_{\bb{Q}_{s}^{x}},$
and using Fubini's theorem, we get
\begin{align}\label{IdenScalCn}
B_{n,k} = \frac{1}{2 \pi }  e^{-s a (k+1)}
\int_{\bb{R}} e^{-i u a (k+1)}   R^{n}_{s, iu}(\varphi_{s,k}^y)(x)
\widehat {\psi}^-_{s,\ee}(u) \widehat\rho_{\ee^{2}}(u) du, 
\end{align}
where
\begin{align*}
R^{n}_{s, iu}(\varphi_{s,k}^y)(x)
= \bb{E}_{\bb{Q}_{s}^{x}} \left[ \varphi_{s,k}^y(G_n \!\cdot\! x) e^{iu T_{n}^x} \right],  \quad x \in \bb P(V). 
\end{align*}
Applying Proposition \ref{KeyPropo} with $l=\frac{ a (k+1) }{n}$ and $\varphi = \varphi_{s,k}^y$, 
we deduce that there exist constants $s_0, c, c_{\ee} > 0$ such that
for any $\ee \in (0,1)$, $s \in (0, s_0)$, $x\in \bb P(V)$, $y \in \bb P(V^*)$, $t \in [n^{\alpha}, o(\sqrt{n} )]$, 
$0 \leq k \leq M_n$ and $\varphi \in \mathscr{B}_{\gamma}$, 
\begin{align*}
&   \int_{\bb{R}} e^{-i u a (k+1)}   R^{n}_{s, iu}(\varphi_{s,k}^y)(x)
\widehat {\psi}^-_{s,\ee}(u) \widehat\rho_{\ee^{2}}(u) du 
  -  \pi_s(\varphi_{s,k}^y)  \frac{ \sqrt{ 2 \pi } }{ s \sigma_s \sqrt{n}}  \nonumber\\
&  \geq  - \left( \frac{c t}{sn} + \frac{c_{\ee}}{sn} + \frac{c}{s t^2 \sqrt{n}} \right) \| \varphi_{s,k}^y \|_{\infty}
      -  c \left( \frac{\log n}{sn} +  e^{-c_{\ee} n } \right) \| \varphi_{s,k}^y \|_{\gamma}. 
\end{align*}
Note that $\sigma_s = \sqrt{ \Lambda''(s) } = \sigma [ 1 + O(s)]$. 
From \eqref{root s 1}, we have $\frac{t}{ s \sigma \sqrt{n} } =  1 + O(s)$, 
and thus $\frac{ t }{ s \sigma_s \sqrt{n} }  = 1 + O(s).$  
Using the inequality $ \sqrt{ 2 \pi } \, t e^{\frac{t^2}{2}} [ 1-\Phi(t) ] \leq 1$ for any $t >0$, 
it follows that 
\begin{align}\label{LowerScalCn}
&   \frac{\int_{\bb{R}} e^{-i u a (k+1)}   R^{n}_{s, iu}(\varphi_{s,k}^y)(x)
\widehat {\psi}^-_{s,\ee}(u) \widehat\rho_{\ee^{2}}(u) du }
{ 2 \pi  e^{\frac{t^2}{2}} [ 1- \Phi(t) ]  } 
  -    \pi_s(\varphi_{s,k}^y)  \frac{  t }{ s \sigma_s \sqrt{n} }   \nonumber\\
&  \geq  - \left( \frac{c t^2}{sn} + \frac{c_{\ee} t}{sn} + \frac{c}{s t \sqrt{n}} \right) \| \varphi_{s,k}^y \|_{\infty}
      -  c \left( \frac{t \log n}{sn} +  e^{-c_{\ee} n } \right) \| \varphi_{s,k}^y \|_{\gamma}   \nonumber\\
&  \geq  - \left( \frac{c t}{\sqrt{n}} + \frac{c_{\ee} }{\sqrt{n}} + \frac{c}{t^2} \right) \| \varphi_{s,k}^y \|_{\infty}
      -  c \left( \frac{\log n}{\sqrt{n}} +  e^{-c_{\ee} n } \right) \| \varphi_{s,k}^y \|_{\gamma}. 
\end{align}
Using the construction of the function $\varphi_{s,k}^y$ and  \eqref{Holder prop ohCHI_k-001-222}
 give that uniformly in $s \in (0, s_0)$, $w \in \bb P(V^*)$, 
$\varphi \in \mathscr{B}_{\gamma}$ and $0 \leq k \leq M_n$, 
\begin{align}\label{Bound_varphi_infty_norm}
\| \varphi_{s,k}^y  \|_{\infty}  \leq c \| \varphi \|_{\infty}
\end{align}
and 
\begin{align}\label{Bound_varphi_Holder_norm}
\| \varphi_{s,k}^y  \|_{\gamma} 
& \leq c \|  \varphi  \|_{\gamma} + c  \frac{e^{ \gamma k a}}{a^{\gamma}} \|\varphi\|_{\infty}.   
\end{align}
Recalling that $t \in [n^{\alpha}, o(\sqrt{n} )]$ and 
 taking $\gamma>0$ sufficiently small such that $ \frac{e^{ \gamma k a}}{a^{\gamma}} < n^{\alpha/2}$,  
 from \eqref{LowerScalCn} we deduce that 
\begin{align*}
\liminf_{n \to \infty} \frac{\int_{\bb{R}} e^{-i u a (k+1)}   R^{n}_{s, iu}(\varphi_{s,k}^y)(x)
\widehat {\psi}^-_{s,\ee}(u) \widehat\rho_{\ee^{2}}(u) du }
{ 2 \pi e^{\frac{t^2}{2}} [ 1- \Phi(t) ]  } 
  \geq    \nu(\varphi_{0,k}^y).  
\end{align*}
As $s=o(1)$ as $n \to \infty$, for any fixed $k \geq 1$ and $a\in (0,\frac{1}{2})$, 
it holds that $\lim_{n \to \infty}e^{-s a (k+1)} = 1$. 
Then, in view of \eqref{IdenScalCn}, it follows that 
\begin{align*}
\liminf_{n \to \infty} \frac{ B_{n,k}}{ e^{ \frac{t^2}{2} } [1- \Phi(t)] }
\geq  \nu(\varphi_{0,k}^y). 
\end{align*}
This, together with \eqref{Unit-partition001-222}, implies the desired lower bound:  
\begin{align}\label{Pf_Lower_bound_B_nk}
 \sum_{k = 0}^{\infty}  \liminf_{n \to \infty} \frac{ B_{n,k}}{ e^{ \frac{t^2}{2} } [1 - \Phi(t)] }
   \mathds{1}_{ \left\{ k \leq M_n \right\}} 
\geq  \sum_{k = 0}^{\infty}  \nu(\varphi_{0,k}^y) = \nu(\varphi). 
\end{align}

\textit{Estimate of $D_{n,k}$.}
We shall apply Fatou's lemma to provide an upper bound for $D_{n,k}$.
An important issue here is to find a dominating function, which is possible due to the integrability of 
the density function $\rho_{\ee^2}$ on the real line. 
More specifically, in the same way as in the proof of \eqref{IdenScalCn}, 
we use the Fourier inversion formula and Fubini's theorem to get 
\begin{align*}
D_{n,k} &  = \frac{1}{2 \pi }  e^{-s a (k+1)}  
 \int_{|w| \geq \ee}  \left[ \int_{\bb{R}} e^{-i u (a (k+1) + w)}   R^{n}_{s, iu}(\varphi_{s,k}^y)(x)
  \widehat {\psi}^-_{s,\ee}(u) \widehat\rho_{\ee^{2}}(u) du  \right]
 \rho_{\ee^2}(w) dw. 
\end{align*}
We decompose the integral in $D_{n,k}$ into two parts:
\begin{align}\label{Pf_Decom_D_nk}
\frac{ D_{n,k} }{ e^{\frac{t^2}{2}} [1 - \Phi(t) ] }  
& =  \frac{1}{2 \pi }  e^{-s a (k+1)}  \left\{ \int_{ \ee \leq |w| < \sqrt{n}} + \int_{|w| \geq \sqrt{n}}  \right\}  \nonumber\\
&  \quad \times  
  \frac{ \int_{\bb{R}} e^{-i u (a (k+1) + w)}   R^{n}_{s, iu}(\varphi_{s,k}^y)(x)
  \widehat {\psi}^-_{s,\ee}(u) \widehat\rho_{\ee^{2}}(u) du  }
   { e^{\frac{t^2}{2}} [1 - \Phi(t) ] } 
    \rho_{\ee^2}(w) dw     \nonumber\\
& =: E_{n,k} + F_{n,k}. 
\end{align}

\textit{Estimate of $E_{n,k}$.}
Since $k \leq M_n$ and $|w| \leq \sqrt{n}$, we have $|l|: = \frac{1}{n} |a (k+1) + w| = O(\frac{1}{\sqrt{n}})$. 
Note that $\frac{ t }{ s \sigma_s \sqrt{n} }  = 1 + O(s)$ and  
$e^{\frac{t^2}{2}} [1 - \Phi(t)] \geq \frac{1}{\sqrt{2 \pi}} (\frac{1}{t} - \frac{1}{t^3})$, $t>1$. 
Applying Proposition \ref{KeyPropo} with $l = (a (k+1) + w)/n$,
we get that uniformly in $\ee \leq |w| < \sqrt{n}$,
\begin{align*}
\lim_{n \to \infty} \frac{ \int_{\bb{R}} e^{-i u (a (k+1) + w)}   R^{n}_{s, iu}(\varphi_{s,k}^y)(x)
  \widehat {\psi}^-_{s,\ee}(u) \widehat\rho_{\ee^{2}}(u) du  }
   { 2 \pi e^{\frac{t^2}{2}} [1 - \Phi(t) ] }   
 =  \nu(\varphi_{0,k}^y).   
\end{align*}
As above, for any fixed $k \geq 1$ and $a\in (0,\frac{1}{2})$, 
we have $\lim_{n \to \infty}e^{-s a (k+1)} = 1$. 
Since the function $\rho_{\ee^2}$ is integrable on $\bb R$, 
by the Lebesgue dominated convergence theorem, we obtain that there exists a constant $c>0$ such that
\begin{align*}
\lim_{n \to \infty} E_{n,k} \mathds{1}_{ \left\{ k \leq M_n \right\}}
=  \nu(\varphi_{0,k}^y) \int_{ |w| \geq \ee } \rho_{\ee^2}(w) dw 
\leq c \ee \nu(\varphi_{0,k}^y). 
\end{align*}
This, together with \eqref{Unit-partition001-222},  implies that
\begin{align}\label{Pf_E_nk1_limit}
\sum_{k = 0}^{\infty} \lim_{n \to \infty} E_{n,k} \mathds{1}_{ \left\{ k \leq M_n \right\}}
\leq  c \ee  \sum_{k = 0}^{\infty}  \nu (\varphi_{0,k}^y)
\leq  c \ee \nu(\varphi). 
\end{align}

\textit{Estimate of $F_{n,k}$.}
By \eqref{Def_Rsz-nth} and \eqref{Contrpsirho},  there exists a constant $c>0$ such that for any $n \geq 1$, 
\begin{align*}
\left| \int_{\bb{R}} e^{-i u (a (k+1) + w)}   R^{n}_{s, iu}(\varphi_{s,k}^y)(x)
  \widehat {\psi}^-_{s,\ee}(u) \widehat\rho_{\ee^{2}}(u) du \right|
& \leq  \frac{c}{s} \| \varphi_{s,k}^y \|_{\infty} \int_{\bb{R}} \widehat\rho_{\ee^{2}}(u) du  \nonumber\\ 
& \leq  \frac{c_{\ee} \sqrt{n}}{t}  \| \varphi_{s,k}^y \|_{\infty}. 
\end{align*}
Using the fact that $t = o(\sqrt{n})$ and $\rho_{\ee^{2}}(w) \leq \frac{c_{\ee}}{w^4}$ for $u \geq 1$, and again the inequality 
$e^{\frac{t^2}{2}} [1 - \Phi(t)] \geq \frac{1}{\sqrt{2 \pi}} (\frac{1}{t} - \frac{1}{t^3})$ for $t>1$,
we get
\begin{align*}
\limsup_{n \to \infty} F_{n,k} \mathds{1}_{ \left\{ k \leq M_n \right\}}
\leq c_{\ee} \limsup_{n \to \infty} \sqrt{n}  \int_{|w| \geq \sqrt{n}}  \rho_{\ee^2}(w) dw 
 = 0.
\end{align*}
Hence 
\begin{align}\label{Pf_E_nk2_limit}
\sum_{k = 0}^{\infty} \limsup_{n \to \infty} F_{n,k} \mathds{1}_{ \left\{ k \leq M_n \right\}}
= 0. 
\end{align}
Since $\ee >0$ can be arbitrary small, 
combining \eqref{Pf_Decom_D_nk}, \eqref{Pf_E_nk1_limit} and \eqref{Pf_E_nk2_limit},
we get the desired bound for $D_{n,k}$:
\begin{align}\label{Pf_Upper_bound_D_nk}
\sum_{k = 0}^{\infty} \limsup_{n \to \infty} 
   \frac{ D_{n,k} }{ e^{ \frac{t^2}{2} } [1-\Phi(t)] }  \mathds{1}_{ \left\{ k \leq M_n \right\}} = 0. 
\end{align}
Putting together \eqref{Pf_Lower_bound_B_nk} and \eqref{Pf_Upper_bound_D_nk},
we conclude the proof of \eqref{ScalLowerBound} as well as the first expansion \eqref{Cramer-Coeffi-upper}. 

The proof of the second expansion \eqref{Cramer-Coeffi-lower}
can be carried out in a similar way. 
As in \eqref{ChanMeaScal} and \eqref{An-lambdaq}, using \eqref{Unit-partition001-222} we have
\begin{align*}
& \bb{E} \left[ \varphi(G_n \!\cdot\! x) 
  \mathds{1}_{ \left\{ \log| \langle f,  G_n v \rangle | - n \lambda_1 \leq  - \sqrt{n}\sigma t \right\} } \right]   \notag\\
 & =   r_s(x) e^{ -\frac{t^2}{2} - \frac{t^3}{ \sqrt{n} } \zeta(-\frac{t}{\sqrt{n}}) }
  \sum_{k = 0}^{ M_n } \bb{E}_{\bb{Q}^{x}_{s}}
\left[ (\varphi r_s^{-1} \chi_{k}^y)(G_n \!\cdot\! x)  e^{-s T_n^x }
\mathds{1}_{\{ T_n^x + Y_n^{x,y} \leq 0 \}}  \right]  \nonumber\\
& \quad +  r_s(x) e^{ -\frac{t^2}{2} - \frac{t^3}{ \sqrt{n} } \zeta(-\frac{t}{\sqrt{n}}) }
  \bb{E}_{\bb{Q}^{x}_{s}}
\left[ (\varphi r_s^{-1} \overline \chi_{M_n + 1}^y)(G_n \!\cdot\! x)  e^{-s T_n^x }
\mathds{1}_{\{ T_n^x + Y_n^{x,y} \leq 0 \}}  \right]  \notag\\
& = : A_n' +  A_n'', 
\end{align*}
where this time we choose $M_n = \floor[]{A \log n}$ with $A>0$, 
and $s < 0$ satisfies the equation \eqref{SaddleEqua-bis}. 
The main difference for handling the first term $A_n'$ consists in using Proposition \ref{KeyPropo-02} instead of Proposition \ref{KeyPropo},
so we omit the details.  
For the second term $A_n''$, we have 
\begin{align}\label{Inequa-Qsx-001}
& \bb{E}_{\bb{Q}^{x}_{s}}
\left[ (\varphi r_s^{-1} \overline \chi_{M_n + 1}^y)(G_n \!\cdot\! x)  e^{-s T_n^x }
\mathds{1}_{\{ T_n^x + Y_n^{x,y} \leq 0 \}}  \right]  \nonumber\\
& \leq  c \|\varphi\|_{\infty} \bb{E}_{\bb{Q}^{x}_{s}}
\left[  \overline \chi_{M_n + 1}^y(G_n \!\cdot\! x)  e^{ s Y_n^{x,y} }  \right]   \nonumber\\
& \leq  c \|\varphi\|_{\infty}  e^{-s a A \log n} \bb{Q}^{x}_{s} \left( - \log \delta(y, G_n \cdot x) \geq a A \log n \right)  \notag\\
& \leq  \frac{c}{n}  \|\varphi\|_{\infty},  
\end{align}
where in the last inequality we use Lemma \ref{Lem_Regu_pi_s} and choose $A > 0$ large enough. 
Using \eqref{Inequa-Qsx-001}, the inequality $t e^{\frac{t^2}{2}} \Phi(-t) \geq \frac{1}{\sqrt{2 \pi}}$ for all $t >0$ and the fact that $t = o(\sqrt{n})$, we get
\begin{align*}
\frac{A_n''}{ e^{  - \frac{t^3}{ \sqrt{n} } \zeta(-\frac{t}{\sqrt{n}}) }  \Phi(-t) } 
\leq  \frac{c}{n}  \|\varphi\|_{\infty} \frac{1}{ e^{\frac{t^2}{2}} \Phi(-t) }  
\leq \frac{ct}{n}  \|\varphi\|_{\infty}  \leq \frac{c}{\sqrt{n}}  \|\varphi\|_{\infty}. 
\end{align*}
This finishes the proof of the expansion \eqref{Cramer-Coeffi-lower}. 
\end{proof}

\begin{proof}[Proof of Theorem \ref{Thm-Cram-Entry_bb}]
Theorem \ref{Thm-Cram-Entry_bb} follows from Theorems \ref{Thm-CramSca-02} and \ref{Thm-Cram-Scalar-tag}. 
\end{proof}

\section{Proof of the local limit theorem with moderate deviations} 

The goal of this section is to establish Theorem \ref{ThmLocal02} on the local limit theorems with moderate deviations
for the coefficients $\langle f, G_n v \rangle$.

The following result which is proved in \cite{XGL19b} will be used to prove Theorem \ref{ThmLocal02}. 
Assume that $\psi: \mathbb R \mapsto \mathbb C$
is a continuous function with compact support in $\mathbb{R}$, 
and that $\psi$ is differentiable in a small neighborhood of $0$ on the real line.

\begin{lemma}[\cite{XGL19b}] \label{Lem_R_st_limit}
Assume conditions \ref{Ch7Condi-Moment} and \ref{Ch7Condi-IP}. 
Then, there exist constants $\gamma, s_0, \delta, c, C >0$ such that for all $s \in (-s_0, s_0)$, 
$x \in \bb P(V)$, $|l|\leq \frac{1}{\sqrt{n}}$, $\varphi \in \mathscr{B}_{\gamma}$ and $n \geq 1$, 
\begin{align} \label{Ine_R_st_limit}
&  \left| \sigma_s \sqrt{n}  \,  e^{ \frac{n l^2}{2 \sigma_s^2} }
\int_{\mathbb R} e^{-it l n} R^{n}_{s, iu}(\varphi)(x) \psi (t) dt
  - \sqrt{2\pi} \pi_s(\varphi) \psi(0) \right|   \nonumber\\
& \leq  \frac{ C }{ \sqrt{n} } \| \varphi \|_\gamma 
  + \frac{C}{n} \|\varphi\|_{\gamma} \sup_{|t| \leq \delta} \big( |\psi(t)| + |\psi'(t)| \big)
  + Ce^{-cn} \|\varphi\|_{\gamma} \int_{\bb R} |\psi(t)| dt. 
\end{align}
\end{lemma}

We also need the result below, which is proved in \cite[Lemma 6.2]{XGL19d}. 

\begin{lemma}[\cite{XGL19d}] \label{Lem_Inte_Regu_a}
Assume \ref{Ch7Condi-Moment} and \ref{Ch7Condi-IP}. 
Let $p>0$ be any fixed constant.
Then, there exists a constant $s_0 > 0$ such that
\begin{align*}
\sup_{n \geq 1} \sup_{ s\in (-s_0, s_0) } \sup_{y \in \bb P(V^*) }  
\sup_{x \in \bb P(V) } 
\bb E_{\bb Q_s^x} \left( \frac{1}{ \delta(y, G_n \!\cdot\! x)^{p|s|} } \right) < + \infty. 
\end{align*}
\end{lemma}

Using Lemmas \ref{Lem_Regu_pi_s},  \ref{Lem_R_st_limit} and \ref{Lem_Inte_Regu_a},
we now prove Theorem \ref{ThmLocal02}. 

\begin{proof}[Proof of Theorem \ref{ThmLocal02}]
Without loss of generality, we assume that the target functions $\varphi$ and $\psi$ are non-negative. 
By the change of measure formula \eqref{Ch7basic equ1}, 
we get that for any $s \in (-s_0, s_0)$ with sufficiently small $s_0 >0$,  
\begin{align}\label{LLT_Def_An_d}
 I :  &=  \mathbb{E} \Big[ \varphi(G_n \!\cdot\! x) \psi \big( \sigma(G_n, x) - n\lambda_1 - \sqrt{n}\sigma t \big) \Big]   \nonumber\\
      & =  r_s(x) \kappa^{n}(s) \mathbb{E}_{ \mathbb{Q}^{x}_{s} }
\Big[ (\varphi r_s^{-1})(G_n \!\cdot\! x) e^{-s \sigma(G_n, x) } 
   \psi \big( \log|\langle f, G_n v \rangle| - n\lambda_1 - \sqrt{n}\sigma t \big) \Big]. 
\end{align}
As in the equation \eqref{SaddleEqua}, for any $|t| = o(\sqrt{n})$ (not necessarily $t > 1$), 
we choose $s \in (-s_0, s_0)$ satisfying the equation
\begin{align}\label{Saddle_Equa_cc}
\Lambda'(s) - \Lambda'(0) = \frac{\sigma t}{\sqrt{n}}. 
\end{align}
Note that $s \in (-s_0, 0]$ if $t \in (-o(\sqrt{n}), 0]$,
and $s \in [0, s_0)$ if $t \in [0, o(\sqrt{n}))$. 
In the same way as in the proof of \eqref{LambdaQ01}, 
from \eqref{Saddle_Equa_cc} it follows that for any $|t| = o(\sqrt{n})$, 
\begin{align*}
s \Lambda'(s) - \Lambda(s) 
= \frac{t^2}{2n} - \frac{t^3}{n^{3/2} } \zeta \Big( \frac{t}{\sqrt{n}} \Big), 
\end{align*}
where $\zeta$ is the Cram\'{e}r series defined by \eqref{Ch7Def-CramSeri}. 
For brevity, denote 
\begin{align*}
T_n^x = \sigma(G_n, x) - n \Lambda'(s),  \qquad  Y_n^{x,y} = \log \delta(y, G_n \!\cdot\! x). 
\end{align*}
Hence, using \eqref{Basic-scalar-decom}, we have
\begin{align*}
I & = r_s(x) e^{-n [s\Lambda'(s) - \Lambda(s)]} \mathbb{E}_{ \mathbb{Q}^{x}_{s} }
\Big[ (\varphi r_s^{-1})(G_n \!\cdot\! x) e^{-s T_n^x } \psi \Big( T_n^x + Y_n^{x,y} \Big) \Big] \nonumber\\
& = r_s(x) e^{- \frac{t^2}{2} + \frac{t^3}{ \sqrt{n} } \zeta( \frac{t}{\sqrt{n} } )} \mathbb{E}_{ \mathbb{Q}^{x}_{s} }
\Big[ (\varphi r_s^{-1})(G_n \!\cdot\! x) e^{-s T_n^x } \psi \Big( T_n^x + Y_n^{x,y} \Big) \Big]. 
\end{align*}
Notice that $r_s(x) \to 1$ as $n \to \infty$, uniformly in $x \in \bb P(V)$. 
Thus in order to establish Theorem \ref{ThmLocal02}, it suffices to prove 
the following asymptotic: as $n \to \infty$, 
\begin{align}\label{LLT_Pf_Object}
 J : = \sigma \sqrt{2 \pi n} \,  \bb E_{\bb Q_s^x}
\Big[ (\varphi r_s^{-1})(G_n \!\cdot\! x) e^{-s T_n^x } \psi \Big( T_n^x + Y_n^{x,y} \Big) \Big]  
 \to \nu(\varphi) \int_{\bb R} \psi(u) du. 
\end{align}
We shall apply Lemmas \ref{Lem_Regu_pi_s} and \ref{Lem_R_st_limit} to establish \eqref{LLT_Pf_Object}.  
Recall that the functions $\chi_k^y$ and $\overline \chi_k^y$ are defined by  \eqref{Def-chi-nk-222}.
Then, using the partition of the unity \eqref{Unit-partition001-222} as in the proof of Theorem \ref{Thm-Cram-Scalar-tag},  we have
\begin{align} \label{PosiScalAn 01}
J  =:  J_1 + J_2, 
\end{align}
where, with $M_n = \floor{A \log n}$ and $A >0$, 
\begin{align*}
&  J_1 = \sigma  \sqrt{2 \pi n} \,  \bb E_{\bb Q_s^x}
  \Big[ (\varphi r_s^{-1} \overline \chi_{M_n}^y)(G_n \!\cdot\! x) e^{-s T_n^x } \psi \Big( T_n^x + Y_n^{x,y} \Big) \Big]  \nonumber\\
& J_2 =  \sigma  \sqrt{2 \pi n}  \,   
 \sum_{k =0}^{ M_n - 1 } \bb E_{\bb Q_s^x}
  \Big[ (\varphi r_s^{-1} \chi_{k}^y)(G_n \!\cdot\! x) e^{-s T_n^x } \psi \Big( T_n^x + Y_n^{x,y} \Big)  \Big].
\end{align*}

\textit{Upper bound of $J_1$.} 
In order to prove that $J_1 \to 0$ as $n \to \infty$, we are led to consider two cases:
$s \geq 0$ and $s<0$.

When $s \geq 0$, since the function $\psi$ has a compact support, say $[b_1, b_2]$,
we have $T_n^x + Y_n^{x,y} \in [b_1, b_2]$. 
Noting that $Y_n^{x,y} \leq 0$, we get $T_n^x \geq b_1$,
and hence $e^{-s T_n^x} \leq e^{-s b_1} \leq c$.
Since the function $\psi$ is directly Riemann integrable on $\bb R$,
it is bounded and hence
\begin{align*}
J_1 \leq  c \sqrt{n}  \,  \bb Q_s^x \Big( Y_n^{x,y} \leq - \floor{A \log n} \Big). 
\end{align*}
Applying Lemma \ref{Lem_Regu_pi_s} with $k = \floor{A \log n}$ 
and choosing $A$ sufficiently large, we obtain that,
as $n \to \infty$, uniformly in $s \in [0, s_0)$, 
\begin{align}\label{Pf_LLTUpp_J1}
J_1 \leq  C \sqrt{n}  \,  e^{-c \floor{A \log n} } \to 0. 
\end{align}

When $s < 0$, 
from $T_n^x + Y_n^{x,y} \in [b_1, b_2]$, we get that
$e^{-s T_n^x} \leq e^{-s b_2 + sY_n^{x,y}}$.
Hence, by H\"{o}lder's inequality, Lemmas \ref{Lem_Regu_pi_s} and \ref{Lem_Inte_Regu_a},
we obtain that, as $n \to \infty$, uniformly in $s \in (-s_0, 0]$,
\begin{align*}
J_1 & \leq  C \sqrt{n}  \, \left\{ \bb E_{\bb Q_s^x} \left( \frac{1}{ \delta(y, G_n \!\cdot\! x)^{-2s} } \right)
 \bb Q_s^x \Big( Y_n^{x,y} \leq - \floor{A \log n} \Big)  \right\}^{1/2}  \nonumber\\
& \leq  C \sqrt{n}  \,  e^{-c \floor{A \log n} }  \to 0,  
\end{align*}
where again $A$ is taken to be large enough.

\textit{Upper bound of $J_2$.} 
Note that the support of the function $\chi_{k}^y$ is contained in the set $\{x \in \bb P(V): -Y_n^{x,y} \in [a (k-1), a(k+1)]\}$.
Therefore on  $\supp \chi_{k}^y$ 
we have  
$- a \leq Y_n^{x,y} + a k \leq a$.
For any $w \in \bb R$ set $\Psi_{s} (w) = e^{-sw} \psi(w)$ and, according to \eqref{smoo001}, define   
${\Psi}^+_{s, \ee}(w) = \sup_{|w - w'| \leq \ee} \Psi_{s} (w')$, for $\varepsilon\in(0,\frac{1}{2})$.
Let 
\begin{align}\label{Def_varphi_skee}
\varphi_{s,k}^y (x) 
   = (\varphi r_s^{-1} \chi_{k}^y)(x),
   \quad  x \in \bb P(V). 
\end{align}  
With this notation, choosing $a \in (0,\ee)$, it follows that 
\begin{align*}
J_2 
 & \leq  \sigma  \sqrt{2 \pi n}     
  \sum_{k =0}^{ M_n -1 }  
\bb E_{\bb Q_s^x}
  \Big[ \varphi_{s,k}^y (G_n \!\cdot\! x) 
  e^{-s Y_n^{x,y} }
    \Psi_{s,\ee}^+ ( T_n^x - a k)  \Big] \\
 & \leq  \sigma  \sqrt{2 \pi n}     
  \sum_{k =0}^{ M_n -1 }  e^{-s a (k-1) }
\bb E_{\bb Q_s^x}
  \Big[ \varphi_{s,k}^y (G_n \!\cdot\! x) 
    \Psi_{s,\ee}^+ ( T_n^x - a k)  \Big].
\end{align*} 
Since the function $\Psi^+_{s, \ee}$ is non-negative and integrable on the real line, 
using Lemma \ref{estimate u convo}, we get
\begin{align*} 
J_2  & \leq  ( 1+ c_{\rho}(\ee) )
\sigma  \sqrt{2\pi n}  \sum_{k =0}^{ M_n -1 }  e^{-s a (k-1)}    \mathbb{E}_{\mathbb{Q}_{s}^{x}}
\left[ \varphi_{s,k}^y (G_n \!\cdot\! x)  
({\Psi}^+_{s, \ee} * \rho_{\ee^2}) (T_n^x - a k) \right], 
\end{align*} 
where $c_{\rho}(\ee) >0$ is a constant  converging to $0$ as $\ee \to 0$. 
Since the function $\widehat\rho_{\ee^{2}}$ is integrable on $\bb{R}$, 
by the Fourier inversion formula, we have 
\begin{align*}
(\Psi^+_{s, \eta, \ee} * \rho_{\ee^2}) (T_n^x - a k)
= \frac{1}{2\pi} \int_{\bb{R}}e^{iu (T_n^x - a k)} 
 \widehat {\Psi}^+_{s, \ee}(u) \widehat \rho_{\ee^2}(u) du. 
\end{align*}
By the definition of the perturbed operator $R_{s,iu}$ (cf. \eqref{Def_Rsz_Ch7}),  
and Fubini's theorem, we obtain
\begin{align} \label{Scal Bnxl 01}
J_2 & \leq  (1+ c_{\rho}(\ee)) \sigma   \sqrt{\frac{n}{2\pi}} 
  \sum_{k =0}^{\infty}  \mathds{1}_{ \{k \leq M_n - 1 \} }  e^{-s a (k-1)}  
   \int_{\mathbb{R}} e^{-iu a k}  R_{s,iu}^n (\varphi_{s,k}^y) (x)
\widehat {\Psi}^+_{s,  \ee}(u) \widehat \rho_{\ee^2}(u) du. 
\end{align}
To deal with the integral in \eqref{Scal Bnxl 01}, we shall use Lemma \ref{Lem_R_st_limit}. 
Note that $e^{ \frac{C k^2}{n} } \to 1$ as $n \to \infty$, uniformly in $0 \leq k \leq M_n -1$. 
Since the function $\widehat {\Psi}^+_{s, \ee} \widehat \rho_{\ee^2}$ is compactly supported on $\bb R$,  
applying Lemma \ref{Lem_R_st_limit} with $\varphi = \varphi_{s,k}^y$, 
$\psi = \widehat {\Psi}^+_{s,  \ee} \widehat \rho_{\ee^2}$
and  $l = \frac{a k}{n \sigma}$, 
we obtain that there exists a constant $c >0$ such that for all $s \in (-s_0, s_0)$, 
$x \in \bb P(V)$, $y \in \bb P(V^*)$, $0 \leq k \leq M_n -1$, $\varphi \in \mathscr{B}_{\gamma}$ and $n \geq 1$, 
\begin{align*}  
& \left| \sigma  \sqrt{\frac{n}{2\pi}}  
  \int_{\mathbb{R}} e^{-it a k} R_{s,iu}^n (\varphi_{s,k}^y) (x)
\widehat {\Psi}^+_{s,  \ee}(u) \widehat \rho_{\ee^2}(u) du   
- \widehat {\Psi}^+_{s,  \ee}(0) \widehat \rho_{\ee^2}(0) \pi_s (\varphi_{s,k}^y) \right| 
 \leq \frac{c}{\sqrt{n}} \| \varphi_{s,k}^y \|_{\gamma}. 
\end{align*}
Using \eqref{Bound_varphi_Holder_norm}
and choosing a sufficiently small $\gamma > 0$, one can verify that 
the following series $\frac{c}{\sqrt{n}} \sum_{k = 0}^{ M_n -1 }  
  \| \varphi_{s,k}^y \|_{\gamma}$
converges to $0$ as $n \to \infty$. 
Consequently, we are allowed to interchange the limit as $n \to \infty$
and the sum over $k$ in \eqref{Scal Bnxl 01}. 
Then, noting that $\widehat\rho_{\ee^{2}}(0) =1$ and 
$\widehat {\Psi}^+_{0, \ee}(0)  = \int_{\mathbb{R}} \sup_{w' \in \mathbb{B}_{\ee}(w)} \Psi_{0}(w') dw$, 
we obtain that uniformly in $v \in V$ and $f \in V^*$ with $\|v\|=1$ and $\| f \| =1$,
\begin{align} \label{LimsuBn a}
\limsup_{n \to \infty}  J_{2}  
& \leq    (1 + c_{\rho}(\ee)) \int_{\mathbb{R}} \sup_{w' \in \mathbb{B}_{\ee}(w)} \Psi_{0}(w') dw
 \sum_{k =1}^{\infty} \nu \left( \varphi_{0,k}^y \right)  \notag\\
 & =    (1 + c_{\rho}(\ee)) \nu(\varphi) \int_{\mathbb{R}} \sup_{w' \in \mathbb{B}_{\ee}(w)} \Psi_{0}(w') dw,
\end{align}
where in the last equality we used \eqref{Unit-partition001-222}. 
Letting $\ee \to 0$, $n \to \infty$, 
and noting that $c_{\rho}(\ee) \to 0$, 
we obtain the desired upper bound for $J_2$: 
uniformly in $v \in V$ and $f \in V^*$ with $\|v\|=\| f \| =1$,
\begin{align} \label{ScaProLimBn Upper 01}
\limsup_{n \to \infty} J_{2}  \leq  \nu(\varphi) \int_{\bb R} \psi(u) du. 
\end{align}

\textit{Lower bound of $J_{2}$.} 
Since on the set $\{-Y_n^{x,y} \in [a (k-1), a(k+1)]\}$, we have
$- a \leq Y_n^{x,y} + a k \leq a$.
Set $\Psi_{s} (w) = e^{-sw} \psi(w)$, $w \in \bb R$
and ${\Psi}^+_{s, \ee}(w) = \inf_{w' \in \mathbb{B}_{\ee}(w)} \Psi_{s} (w')$, for $\varepsilon\in(0,\frac{1}{2})$.  
Then, with $a \in (0, \ee)$,
\begin{align*}
J_2 & \geq  \sigma  \sqrt{2 \pi n}     
  \sum_{k =0}^{ M_n-1 }
\bb E_{\bb Q_s^x}
  \Big[ \varphi_{s,k}^{y}(G_n \!\cdot\! x) e^{-s Y_n^{x,y} } \Psi_{s,\ee}^- ( T_n^x - a k) \Big]  \nonumber\\
 & \geq  \sigma  \sqrt{2 \pi n}     
  \sum_{k =0}^{ M_n-1 }  e^{-s a k}
\bb E_{\bb Q_s^x}
  \Big[ \varphi_{s,k}^{y}(G_n \!\cdot\! x) 
    \Psi_{s,\ee}^- ( T_n^x - a k)  \Big].
\end{align*} 
By Fatou's lemma, it follows that
\begin{align*} 
\liminf_{n \to \infty} J_2  
\geq  \sum_{k =0}^{\infty} \liminf_{n \to \infty}  \sigma   \sqrt{2 \pi n}      
   e^{-s a k}   \mathbb{E}_{\mathbb{Q}_{s}^{x}}
\left[  \varphi_{s,k}^y (G_n \!\cdot\! x)  \Psi^-_{s, \ee} (T_n^x - a k)\right]. 
\end{align*}
Since $s=o(1)$ as $n\to\infty$, 
we see that for fixed $k \geq 1$, we have $e^{-s a k} \to 1$ as $n \to \infty$. 
Since the function $\Psi^-_{s, \ee}$ is non-negative and integrable on the real line, 
by Lemma \ref{estimate u convo}, we get
\begin{align} \label{ScaProLimAn Bn 0134}
 \liminf_{n \to \infty} J_2  
& \geq  \sum_{k =0}^{\infty}  \liminf_{n \to \infty}  \sigma \sqrt{2 \pi n}
\bb E_{\bb Q_s^x}
  \Big[ \varphi_{s,k}^y (G_n \!\cdot\! x)
    (\Psi^-_{s,  \ee} * \rho_{\ee^2}) (T_n^x - a k)  \Big]  \nonumber\\
&  \quad  -  \sum_{k =0}^{\infty} \limsup_{n \to \infty}  \sigma  \sqrt{2 \pi n}
  \int_{|w| \geq \ee}  \mathbb{E}_{\mathbb{Q}_{s}^{x}}  
  \left[  \varphi_{s,k}^y (G_n \!\cdot\! x) \Psi^-_{s, \ee}  (T_n^x - a k - w)
   \right] \rho_{\ee^2}(w)  dw  \nonumber\\
& =: J_3 - J_4. 
\end{align}

\textit{Lower bound of $J_3$.} 
Proceeding as in the proof of the upper bound \eqref{LimsuBn a} and using Lemma \ref{Lem_R_st_limit},
we obtain 
\begin{align} \label{Low_Boun_J3}
  J_3  \geq  \int_{\mathbb{R}} \inf_{w' \in \mathbb{B}_{\ee}(w)} \Psi^-_{0, \ee}(w') dw
 \sum_{k =0}^{\infty} \nu \left(  \varphi_{0, k}^y \right).  
\end{align}
Taking the limit as $\ee\to 0$ and $n\to\infty$,
we get the lower bound for $J_3$: 
uniformly in $v \in V$ and $f \in V^*$ with $\|v\|=1$ and $\| f \| =1$, 
\begin{align} \label{ScaProLimBnLow_J3}
\liminf_{n \to \infty}  J_3 \geq  \nu(\varphi) \int_{\bb R} \psi(u) du. 
\end{align}

\textit{Upper bound of $J_4$.} 
By Lemma \ref{estimate u convo}, we have
$\Psi^-_{s, \ee} \leq (1+ c_{\rho}(\ee)) \Psi^+_{s, \ee} * \rho_{\ee^2}$. 
Applying Lemma \ref{Lem_R_st_limit} with $\varphi = \varphi_{s,k}^y$ 
and $\psi = \widehat{\Psi}^+_{s, \ee} \widehat{\rho}_{\ee^2}$,
it follows from the Lebesgue dominated convergence theorem that  
\begin{align}\label{Pf_Lower_J41}
J_{4}  \leq (1 + c_{\rho}(\ee) ) \sum_{k=0}^{ \infty }
\nu \left( \varphi_{0,k}^y \right)  \widehat{\Psi}^+_{0, \ee}(0) \widehat{\rho}_{\ee^2}(0)
\int_{|w| \geq \ee} \rho_{\ee^2}(w) dw, 
\end{align}
which converges to $0$ as $\ee \to 0$.

Combining \eqref{ScaProLimAn Bn 0134}, \eqref{ScaProLimBnLow_J3} and \eqref{Pf_Lower_J41}, 
we get the desired lower bound for $J_2$: 
uniformly in $v \in V$ and $f \in V^*$ with $\|v\|=1$ and $\| f \| =1$, 
\begin{align} \label{lowelast001}
\liminf_{n\to\infty} J_2 \geq  \nu(\varphi) \int_{\bb R} \psi(u) du.
\end{align}
Putting together
\eqref{PosiScalAn 01}, \eqref{Pf_LLTUpp_J1}, \eqref{ScaProLimBn Upper 01} and \eqref{lowelast001},
we obtain the asymptotic \eqref{LLT-Moderate-01}. 
This ends the proof of Theorem \ref{ThmLocal02}. 
\end{proof}

\textbf{Acknowledgment.}
The work has been supported by the National Natural Science Foundation of China 
(Grants No. 11971063, No. 11571052 and No. 11731012).
The work has also benefited from the support of the Centre Henri Lebesgue (CHL, ANR-11-LABX-0020-01). 


\end{document}